\documentclass[preprint,10pt,3p]{elsarticle}
\usepackage{graphicx,amsmath,amsthm,amssymb,latexsym,amsfonts,color}
\usepackage[bookmarksnumbered, colorlinks, plainpages]{hyperref}
\usepackage{xcolor}
\usepackage{sectsty}
\usepackage{comment}
\usepackage{tabularx,booktabs}
\usepackage{algorithm2e}
\usepackage{booktabs}
\usepackage{hhline,booktabs}
\usepackage{multirow}
\usepackage{soul}
\usepackage{footnotehyper}
\usepackage{tikz}
\usepackage{natbib}
\usepackage{pifont}

\usepackage[flushleft]{threeparttable}
\usepackage{array,booktabs,makecell}
\usepackage[font=small]{caption}

\definecolor{astral}{RGB}{46,116,181}
\sectionfont{\color{astral}}

\newtheorem{thm}{Theorem}[section]

\newtheorem{prop}[thm]{Proposition}
\newtheorem{defn}[thm]{Definition}
\newtheorem{rem}[thm]{\bf{Remark}}

\definecolor{fgreen}{rgb}{0.13, 0.55, 0.13}

   \newcommand{\Jj}{\mathbb{J}}
\newcommand{\Cc}{\mathbb{C}}
\newcommand{\Rr}{\mathbb{R}}
\newcommand{\Mm}{\mathbb{M}}

\newcommand{\Aa}{\mathfrak{a}}
\newcommand{\BB}{\mathfrak{b}}

\newcommand{\DD}{\mathfrak{d}}
\newcommand{\GG}{\mathfrak{g}}
\newcommand{\HH}{\mathfrak{h}}
\newcommand{\JJ}{\mathfrak{j}}
\newcommand{\GKK}{\mathfrak{k}}

\newcommand{\GL}{\mathfrak{gl}}
\newcommand{\SL}{\mathfrak{sl}}

\newcommand{\SO}{\mathfrak{so}}
\newcommand{\SP}{\mathfrak{sp}}

\newcommand{\PP}{\mathfrak{p}}
\newcommand{\RR}{\mathfrak{r}}
\newcommand{\LT}{\mathfrak{lt}}
\newcommand{\SLT}{\mathfrak{slt}}
\newcommand{\SUT}{\mathfrak{sut}}
\newcommand{\UT}{\mathfrak{ut}}
\newcommand{\WW}{\mathfrak{w}}

\DeclareMathOperator{\Tr}{Tr}
\DeclareMathOperator{\Diag}{diag}

\DeclareMathOperator{\Det}{det}

\begin{document}
	
	\date{}
	

\begin{frontmatter}\title{A Lie algebra view of matrix splittings\\
\bigskip
                  {\bf {\small Dedicated to Daniel Szyld on the occasion of his 70th birthday}}}
		
		\author[add1]{Michele Benzi}
		\author[add2]{Milo Viviani}

		\address[add1]{Scuola Normale Superiore, Piazza dei Cavalieri, 7, 56126, Pisa, Italy {\tt (michele.benzi@sns.it)}.}
		\address[add2]{Scuola Normale Superiore, Piazza dei Cavalieri, 7, 56126, Pisa, Italy {\tt (milo.viviani@sns.it)}.}

		\begin{abstract}
		In this paper we use some basic facts from the theory of (matrix) Lie groups and algebras to show that many of the classical
		matrix splittings used to construct stationary iterative methods and preconditioniers for Krylov subspace methods can be interpreted as 
		linearizations of matrix factorizations. Moreover, we show that
		new matrix splittings are obtained when we 
		specialize these splittings to some of the classical matrix groups and their Lie and Jordan algebras. As an example,  we derive structured generalizations
	        of the HSS (Hermitian and skew-Hermitian splitting) iteration, and provide sufficient conditions for their convergence.
	        \end{abstract}
		

	\end{frontmatter}

	\section{Introduction}\label{sec:1}  Many classical algorithms of numerical linear algebra are based on matrix decompositions. These 
	decompositions can be either multiplicative or additive; following standard terminology, herein we will refer to the former type as  
	\textit{matrix factorizations}, and to the latter one as 
	\textit{matrix splittings}. \\
	
		Matrix factorizations are at the basis of many of the most efficient algorithms for solving linear systems and for computing eigenvalues (or 
	singular values) of matrices. The LU, Cholesky, and QR (or LQ) factorizations are among the best known matrix decompositions \cite{GVL4}.  Additive
	decompositions, or splittings, 
	of the form $A=B+C$ (frequently written in the form $A=B-C$) are used to define stationary iterative methods, including 
	the classical Jacobi and Gauss-Seidel iterations \cite{Varga}, and to construct preconditioners used in conjunction with Krylov subspace methods \cite{Saad}.
	Splitting methods are also widely used in the numerical solution of systems of ODEs, particularly on manifolds,  and semi-discretized PDEs
	\cite{MacLachlan2010}.  \\
	
	It has long been recognized that many matrix factorizations can be interpreted in terms of decompositions of a matrix group as a product
	of subgroups.\footnote{One of the first papers drawing the attention
	of numerical analysts to the relation of matrix factorizations to group decompositions is \cite{DellaDora}.}
	Matrix groups endowed with additional (topological or differentiable) structure are fundamental objects in many areas of mathematics,
	from Lie theory to numerical analysis \cite{Iserles2000}, and the connection between matrix factorizations and different representations of  a matrix group as a product
	of subgroups is fairly standard, though not widely used in numerical linear algebra. \\
	
	In this paper we show that just as many matrix factorizations can be understood in terms of (multiplicative) decompositions of a matrix Lie group as a product of subgroups,
	certain matrix splittings can similarly be understood in terms of (additive) decompositions of the corresponding Lie algebra (the tangent space of the
	Lie group at the identity).  In some cases the Lie algebra of the group splits as a sum
	of Lie subalgebras which corresponds to subgroups of the matrix group. In others, as we will see,
	the Lie algebra splits as the sum of a Lie subalgebra and a Jordan algebra. This second situation arises when one of factors in the matrix factorization
	comes from a subgroup while the other comes from a vector space of matrices. 
	 Moreover, many additive splitting can be obtained as the 
	linearization of an underlying multiplicative factorization of the underlying matrix group. Hence, many classical matrix splittings appear in a new 
	light: they are \textit{infinitesimal counterparts} of well-known matrix factorizations. To  the best of our knowledge this observation, while simple, 
	has not been explicitly made before in the numerical analysis literature. \\
	
	To informally describe what this paper is about, we recall that 
	passing from a matrix Lie group $G$ to the tangent space at the identity (i.e., linearizing $G$ near $I_n$) results in a vector space $\GG$  of matrices which, endowed with the
	commutator product  $[A,B] = AB-BA$, becomes a Lie algebra.  Consider now a smooth path of matrices $t\mapsto A(t)$ in $G$ with $A(0) = I_n$
	and a smooth factorization 
        \begin{equation}\label{decomp}
	A(t) = B(t)C(t),
	\end{equation}
	where $B(t)$ and $C(t)$ are smooth paths entirely lying in submanifolds  $H, K\subset G$ for $|t|$ small and such that
	$B(0) = C(0) = I_n$. Differentiating (\ref{decomp}) for  $t=0$ yields
	\begin{equation}\label{split}
	A'(0) = B'(0) C(0) + B(0)C'(0) = B'(0) + C'(0)\,,
	\end{equation}
	and we see that the tangent vector $A'(0)$ to $G$ at the identity, an element of the Lie algebra $\GG$ of $G$, decomposes as the sum of two matrices $B'(0$) and $C'(0)$, one tangent to 
	$H$ and the other tangent to $K$ at the identity. In other words, the Lie algebra $\GG$ splits as the sum of two vector subspaces, which in turn may be endowed 
	with additional algebraic structure. For instance, if $H$ and $K$ are matrix subgroups of $G$, the corresponding tangent spaces $\HH$ and $\GKK$ are themselves Lie algebras;
	if $H\cap K = \{I_n\}$, then $\HH \cap \GKK = \{0\}$ and we obtain that $\GG$, as a vector space, admits the direct sum decomposition  $\GG = \HH \oplus \GKK$. \\
	
	When $G = GL(n, \Rr)$, the general linear group, the corresponding Lie algebra $\GG$, denoted with $\GL_n$, is just the total matrix space $\Rr^{n\times n}$ 
	endowed with the commutator product. Any mutiplicative decomposition of $G$ into a product of submanifolds leads to a corresponding additive splitting of $\GL_n$,
	and the structure of the matrices $B$ and $C$ in the splitting $A=B+C$ will depend on the choice of $H$ and $K$. For instance, 
	linearization of the \textit{polar factorization} of
	$GL(n,\Rr)$  leads to the \textit{symmetric and skew-symmetric splitting} of a generic $n\times n$ matrix, an instance of the \textit{Cartan decomposition};
	see subsection \ref{sec:Cartan}. 
	Furthermore, this splitting, widely used in numerical linear algebra 
	(e.g.,  \cite{Bai,BaiBook,Benzi2004,Concus,Szyld2022,Widlund}), corresponds to a decomposition of $\GL_n$ into subspaces that are 
	mutually orthogonal with respect to the Frobenius inner product. One of these subspaces carries a Lie algebra structure, the other one a Jordan algebra
	structure. This provides  us with a canonical example of what we call here a \textit{Lie-Jordan splitting}; see Section \ref{sec:4}. \\
	
	In this paper we consider several different group decompositions and derive the corresponding matrix splittings.
	 In some cases, instead of considering a global factorization of $G$, it will be convenient to localize the discussion near the identity $I_n$, i.e., to
	 consider a decomposition of an open neighborhood of the identity matrix. Since the 
	linearization $\GG$ of $G$ consists of all the tangent vectors at the identity, this still leads to an additive splitting for every matrix in $\GG$.\\
	
	Instead of starting from a multiplicative  group decomposition in order to obtain a class of additive matrix splittings via linearization, it is also possible to proceed in the reverse order:
	starting from a matrix splitting, we can try to determine the corresponding matrix factorization, or group decomposition. This passage from local to global can be interpreted as
	integration, as opposed to differentiation, and just as indefinite integration leads to a non-unique solution (up to additive constants), there may be more than one group factorization
	leading to the same splitting of the tangent space (Lie algebra). Another way to see this is to recall that while the Lie algebra of a Lie group is uniquely determined, 
	being the tangent space  to the Lie group at the identity, there can be
	more than one Lie group with the same Lie algebra: for instance, the orthogonal group $O(n,\Rr)$ and the special orthogonal group 
	$SO(n,\Rr)$ both have $\SO_n$, the algebra of skew-symmetric matrices, as their Lie algebra. 
	In the following, sometimes we will start from a matrix factorization and derive the corresponding matrix splitting, while in others we start from the splitting, regarded as an additive
	decomposition of a matrix Lie algebra, and then discuss one or more group decompositions, i.e., matrix factorizations,
	the linearization of which results in the additive splitting.   \\
	
	Classical stationary iterative 
	methods, like Jacobi and Gauss-Seidel, are based on the splittings of $A$ in terms of the diagonal, strictly lower, and strictly upper triangular parts \cite{Varga}.
	As we will see, these classical splittings are infinitesimal counterparts, or linearizations, of the LU or (LDU) factorization. 
	Here again we find that linearization leads to a splitting of the Lie algebra of all $n\times n$ matrices into a direct sum of subspaces, which in the case of 
	Gauss-Seidel is an orthogonal decomposition relative to the
	Frobenius inner product.\\
	
	 We also examine other matrix splittings, for example those corresponding to 
	 the QR factorization (a special case of the \textit{Iwasawa decomposition} of a Lie group, see \cite{Helgason,Knapp}), as well as others which generalize the 
	 symmetric and skew-symmetric
	 splitting.\\
	 
	 The two main contributions of the paper are the following. First, we provide a link (of a geometric nature) between some of the classical matrix decompositions and the corresponding 
	 matrix splittings, which has been hitherto overlooked in the numerical linear algebra literature; and second, we present new matrix splittings, in particular some
	 generalizations of the well-known HSS method \cite{Bai}, which may prove useful in the construction of iterative solvers and preconditioners for block linear systems. 
	Some of these splittings may also find appplication, for example, 
	in developing structure-preserving ODE or eigenvalue solvers. \\		
		
	The remainder of the paper is organized as follows. In Section \ref{sec:2} we provide basic background information on Lie groups, Lie algebras and Jordan algebras,
	primarily restricted to the classical (matrix) groups and their algebras.
	In Section \ref{sec:3} we survey several types of matrix splittings and interpret them as linearizations of  well-known matrix factorizations.  
	In Section \ref{sec:4} we discuss the class of Lie--Jordan splittings
	and in Section \ref{sec:LJ} we establish a general convergence criterion for what we call the $J$-HSS method
	(a generalization of the standard HSS iteration \cite{Bai}), and provide examples of this method.  
	Section \ref{sec:6} contains brief discussions of two alternating-type
	iterations and the corresponding splittings.  Some conclusions and suggestions for future work are given in Section \ref{sec:7}.

	\section{Preliminaries and background}\label{sec:2}    
		In this Section we recall some basic notions about matrix Lie groups, Lie algebras and Jordan algebras, with the aim of introducing notations and making the paper
	as self-contained as possible.  For ease of exposition, in this paper we limit ourselves to real matrices, but much of the material in the
	paper can be extended to the complex case. Our treatment is elementary and informal; we refer the reader to, e.g., 
	\cite{Helgason,Knapp,KN,McCrimmon,Postnikov} for thorough and general accounts of the theory of Lie groups and algebras, differentiable manifolds,
	and Jordan algebras.\\   
	
	Let $GL(n,\Rr)$ be the general linear group of degree $n$ over $\Rr$, i.e., the set of all real invertible $n\times n$ matrices 
	endowed with the standard matrix product.  We note that $GL(n,\Rr)$ is both a group and a differentiable manifold (of dimension $n^2$),
	and that the mappings  $(A,B)\mapsto AB$ and $A\mapsto A^{-1}$ are smooth. These facts can be condensed in the statement that $GL(n,\Rr)$ 
	is a (matrix) Lie group.\\
	
	The tangent space to $GL(n,\Rr)$ at the identity $I_n$ 
	coincides with
	the whole space $\Rr^{n\times n}$.  To see this, take $A\in \Rr^{n\times n}$, consider (for example) the smooth path $t\mapsto A(t) \in GL(n,\Rr)$ defined by $A(t) = \exp(tA)$,
	and observe that $A(0)=I_n$, $A'(0) = A$. Hence, every $A\in \Rr^{n\times n}$ is a tangent vector to $GL(n,\Rr)$ at the identity.\\
	
	With respect to the usual matrix operations, $\Rr^{n\times n}$ is an associative, non-commutative algebra with unit.  
	However, we can also introduce two other types of products in $\Rr^{n\times n}$: the commutator product $[A,B] = AB-BA$, and the Jordan
	product $A\bullet B=\frac{1}{2}(AB+BA)$.  The commutator product does not admit an identity element, is anticommutative ($[B,A] = - [A,B]$), 
	nilpotent ($[A,A] = 0$) 
	and is non-associative; however, it satisfies a weak version of associativity, the {\textit {Jacobi identity}}:
	$$[[A,B], C] + [[C,A], B] + [[B,C],A] = 0 \quad \forall A, B, C, \in \Rr^{n\times n}.$$
	
	The Jordan product admits $I_n$ as identity element, is commutative, satisfies $A\bullet A = A^2$, and is non-associative; it satisfies a weak version
	of associativity, the {\textit {Jordan identity}}:
	$$(A\bullet B)\bullet (A\bullet A) = A\bullet (B\bullet (A\bullet A)) \quad \forall A, B \in \Rr^{n\times n}.$$
	
	We express these properties by saying that $\Rr^{n\times n}$ is a {\textit{Lie algebra}} with respect to the commutator product, and a {\textit{Jordan algebra}}
	with respect to the Jordan product.  We will denote the Lie algebra of all real $n\times n$ matrices with the commutator product by $\GL_n$, and the Jordan
	algebra of all real $n\times n$ matrices with the commutator product by $\Mm(n,\Rr)$.\\
	
	We will be dealing with vector subspaces of $\Rr^{n\times n}$ which are also Lie subalgebras or Jordan subalgebras. Prominent Lie subalgebras
	of $\GL_n$ are $\SL_n$, the Lie algebra of
	all $n\times n$ matrices with zero trace, and $\SO_n$, the Lie algebra of all $n\times n$  skew-symmetric matrices: note that 
	$\SO_n \subset \SL_n \subset \GL_n$.  The Lie algebra $\SL_n$ is the tangent space at the identity of the {\textit{special linear group}} $SL(n,\Rr)$,
	consisting of all the real $n\times n$ matrices with unit determinant, while $\SO_n$ is the Lie algebra of the 
	{\textit{special orthogonal group}} $SO(n,\Rr)$, which consists of all the real $n\times n$ orthogonal matrices with unit determinant. This group is the
	connected component (containing the identity) of the group $O(n,\Rr)$ of all  the real $n\times n$ orthogonal matrices, therefore $\SO_n$ is also the Lie algebra (tangent space) of $O(n,\Rr)$.
	The subgroups $O(n,\Rr)$ and $SO(n,\Rr)$ are topologically closed subgroups of the Lie group $GL(n,\Rr)$ and therefore they are themselves Lie groups.
	The Lie algebra of a Lie group is often thought of as the {\textit{linearization}} of the Lie group near the identity.\\
	
	In the course of the paper we will also deal with groups and algebras of triangular and diagonal matrices.  The groups $LT(n,\Rr)$ and $UT(n,\Rr)$ of all invertible
	lower (respectively, upper) triangular matrices are Lie groups; their Lie algebras are easily seen to be the (not necessarily invertible) lower and upper triangular matrices,
	denoted here by $\LT_n$ and $\UT_n$, respectively. Important subgroups of $LT(n,\Rr)$ are $LT_+(n,\Rr)$ (lower triangular matrices with positive diagonal
	entries), $SLT(n,\Rr) = SL(n,\Rr)\cap LT(n,\Rr)$ (lower triangular matrices with unit determinant), and $LT_1(n,\Rr)$ (unit lower triangular matrices).  
	The subgroups $UT_+(n,\Rr)$, $SUT(n,\Rr)$ and $UT_1(n,\Rr)$ of $UT(n,\Rr)$ are defined analogously.  The Lie algebra of $LT_+(n,\Rr)$ is easily seen to be again $\LT_n$,
	the Lie algebra of $SLT(n,\Rr)$ is $\SLT_0 = \SL_n\cap \LT_n$ (the lower triangular matrices with zero trace), and 
	the Lie algebra of $LT_1(n,\Rr)$ is the nilpotent Lie subalgebra of the strictly lower triangular matrices, $\LT_0\subset \LT_n$. The Lie algebra of upper triangular matrices
	with zero trace will be denoted by $\SUT_0$ and the nilpotent Lie algebra of all strictly
	upper triangular matrices  by $\UT_0$.\\
	
	We will also use $\DD$ and $\DD_0$ to denote the Lie algebras of all diagonal matrices and diagonal matrices with zero trace, respectively. They are the Lie algebras
	of the groups of all invertible diagonal matrices (and of its subgroup of all diagonal matrices with positive diagonal entries) and of the group of diagonal matrices with unit
	determinant, respectively.\\
		
	 The most important example of a Jordan subalgebra of $\Mm(n,\Rr)$  is provided by the space of all real $n\times n$ symmetric matrices.  Indeed, the Jordan
	 product was introduced (by German quantum physicist Pascual Jordan) as a way to ``symmetrize" the product of two possibly non-commuting symmetric 
	 operators.\footnote{Note that $A\bullet B = AB$ whenever $AB=BA$.}  We denote this Jordan subalgebra by $\PP_n$.\\
	 
	 The real orthogonal group $O(n,\Rr)$ and its subgroup $SO(n,\Rr)$ are examples of  {\textit{classical groups}}.  The group $O(n,\Rr)$ can be characterized as the
	 group of linear transformations $Q$ on $\Rr^n$ such that $\langle Qx, Qy\rangle = \langle x, y\rangle$ for all $x,y \in \Rr^n$, where $\langle x, y\rangle = x^Ty$ denotes
	 the standard (Euclidean) inner product.  Other examples of classical groups are the {\textit {pseudo-orthogonal group}} $O(p,q,\Rr)$ (where $p+q=n$) of all linear
	 transformations on $\Rr^n$  which preserve the indefinite inner product $x^T I_{p,q}y$, where $I_{p,q} = \Diag(I_p, -I_q)$, and the  {\textit {symplectic group}} $Sp(n,\Rr)$ of all 
	 linear transformations on $\Rr^{2m}$ which preserve the symplectic form $x^T \Sigma_n y$, where $n=2m$ and
	  $$\Sigma_n = \left[ {\begin{array}{*{20}{c}}
				0 & I_m\\
				-I_m& 0
		\end{array}}\right ] \,.$$
	We also have the group of pseudo-orthogonal matrices with unit determinant, $SO(p,q,\Rr)$.  (Note that symplectic matrices always have unit determinant.)  The Lie algebra of
	the group $O(p,q,\Rr)$ (and also of $SO(p,q,\Rr)$) consists of all the pseudo-skew-symmetric matrices, while the Lie algebra of $Sp(n,\Rr)$ consists of all the Hamiltonian 
	matrices; these algebras are described in Section \ref{sec:4}, where we also describe the corresponding Jordan algebras, namely,  the pseudo-symmetric and skew-Hamiltonian matrices,
	respectively.\\

        A  connected Lie group that does not contain any non-trivial, connected, Abelian normal subgroup is said to be {\textit{a semisimple Lie group}}.  The groups $GL(n,\Rr)$
        and $O(n,\Rr)$ are not connected and therefore they are not semisimple.    The connected Lie group $SO(n,\Rr)$, on the other hand, is semisimple, as is the symplectic 
        group $Sp(n,\Rr)$. \\
       
        The corresponding notion for a Lie algebra $\GG$  is defined as follows.  If $\Aa$ and $\BB$ are vector subspaces of $\GG$, denote by $[\Aa,\BB]$ the subspace generated by the 
        elements of the form $[A,B]$ with $A\in \Aa$ and $B\in \BB$. The property of $\Aa$ being a Lie subalgebra of $\GG$ can then be expressed as $[\Aa,\Aa] \subseteq \Aa$.
        The condition $[\Aa,\GG]\subseteq \Aa$ expresses the property of $\Aa$ being a {\textit{Lie ideal}} of $\GG$. As an example, the space $\LT_0$ of strictly lower triangular 
       matrices is a Lie ideal of the Lie  algebra $\LT_n$ of all lower triangular matrices.  Clearly, any Lie ideal is a Lie subalgebra of $\GG$.  A Lie subalgebra $\Aa$
        is $Abelian$ if $[A,B] = 0$ for all $A, B\in \Aa$.  A Lie algebra that does not contain any non-trivial Abelian ideals is said to be a {\textit{semisimple Lie algebra}}. A basic result
        in Lie theory is that {\textit{a Lie group is semisimple if and only if its Lie algebra is semisimple}}. Hence, $\SL_n$ is semisimple, while $\GL_n$ is not: the set of all scalar matrices 
        $k I_n$ ($k\in \Rr)$ is a nontrivial Abelian ideal in $\GL_n$.  Denoting by $\RR$ this ideal, the Lie algebra $\GL_n$ decomposes as
        \begin{equation}\label{Levi}
        \GL_n = \SL_n \oplus \RR\,.
        \end{equation}
        The decomposition (\ref{Levi}) is known as the {\textit{Levi decomposition}} of $\GL_n$.  More generally, every Lie algebra $\GG$  admits a Levi decomposition,
        i.e., it can be written as the direct sum of a semisimple Lie subalgebra and the {\textit{radical}} $\RR$ of $\GG$ (we refer to \cite{Helgason} or \cite{Postnikov} for the definition
        of the radical of a Lie algebra). The importance of the Levi decomposition is that it can be used to determine the structure of Lie groups and algebras by reducing it to the semisimple case,
        which is much more tractable.

	\section{Matrix factorizations and matrix splittings}\label{sec:3}
	
	In this Section we discuss several matrix splittings of interest in numerical linear algebra and describe how they correspond to linearizations of well known
	matrix factorizations. 

         \subsection{Cartan decomposition} \label{sec:Cartan}
         One of the most important decompositions of a Lie algebra (and of the corresponding Lie  group) is the \textit{Cartan decomposition}; see for example
         \cite{Helgason,Knapp}.  Rather than giving the general definition, here we describe it in the special cases we are mostly interested in.  \\
         
         Let us consider the map $\theta: \GL_n \rightarrow \GL_n$ defined by $\theta (X) = - X^T$.  Then $\theta$ is a Lie algebra automorphism which satisfies 
         $\theta^2 = \text{Id}$, i.e., it is an involution. Hence, $\theta$ has two eigenvalues, $+1$ and $-1$. The corresponding eigenspaces are, respectively, the Lie
         subalgebra of all skew-symmetric matrices $\SO_n$ and the subspace $\PP_n$ of all $n\times n$ real symmetric matrices. This results in the decomposition
         \begin{equation}\label{Cartan_alg}
         \GL_n =  \SO_n \oplus \PP_n\,,
         \end{equation}
         which is a Cartan decomposition of $\GL_n$.\footnote{Here we follow the example of \cite{Edelman2} and commit a slight abuse of language:  strictly speaking,
         the term \textit{Cartan decomposition}  refers to the corresponding decomposition of the semisimple Lie algebra  $\SL_n
         =\SO_n \oplus \PP_0$, where $\PP_0$ is the subspace of all real symmetric matrices with zero trace.  Since $\SO_n$ is the Lie algebra corresponding to both $O(n,\Rr)$ and  $SO(n,\Rr)$, this abuse is inconsequential for our
         purposes. Furthermore, recalling the Levi decomposition $\GL_n = \SL_n \oplus \RR$ and observing that $\PP_n= \PP_0 \oplus \RR$, we conclude
         that (\ref{Cartan_alg}) is equivalent to the ``genuine" Cartan decomposition $\SL_n=\SO_n \oplus \PP_0$.
        }
         Note that (\ref{Cartan_alg}) is simply the statement that any $n\times n$ real matrix $A$ can be written uniquely as the sum of a skew-symmetric and a symmetric matrix:
         \begin{equation}\label{HSS}
	A = \mathcal{S}(A) + \mathcal{H}(A)\,,
	\end{equation}
	where $\mathcal{S}(A) :=\frac{1}{2}(A - A^T)\in \SO_n$ is the skew-symmetric part of $A$ and 
	$\mathcal{H}(A):=\frac{1}{2}(A +A^T)\in \PP_n$ the symmetric part. However, there is something more to this.  First of all, 
	the two subspaces  $\SO_n$ and $\PP_n$ are mutually orthogonal
	with respect to the Frobenius inner product on $\Rr^{n\times n}$, defined as $\langle A, B\rangle=\Tr(A^TB)$, as is easily checked
	($\Tr(\cdot)$ denotes the trace). Hence, the direct sum (\ref{Cartan_alg}) is actually an
	orthogonal sum. In other words, $\SO_n$ is the tangent space to $SO(n,\Rr)$ at the identity, while $\PP_n$ is the normal space.
	Second, while $\SO_n$ is a Lie subalgebra of $\GL_n$ (and indeed of $\SL_n$), $\PP_n$ is not a Lie algebra; it is, however, 
	a Jordan algebra with respect to the Jordan product $A\bullet B = \frac{1}{2}(AB + BA)$.  \\
	
	As already mentioned in the Introduction, the symmetric and skew-symmetric splitting is the infinitesimal version of the polar decomposition $A=QP$, where $Q$ is orthogonal
	and $P$ is symmetric and positive definite; since $A$ is nonsingular, this factorization is unique.  However, some care is needed here in order to have a 
	precise statement of this fact.  First we restrict our attention
	to the (semisimple) Lie group of all real $n\times n$ matrices with unit determinant,
	$SL(n,\Rr)$; its Lie algebra, $\SL_n$, consists of all real $n\times n$ matrices with zero trace. 
	The \textit{global Cartan decomposition} of the
	special linear group is 
	\begin{equation}\label{Cartan_group_1}
         SL(n,\Rr) =  SO(n,\Rr) \cdot \exp ({\PP_0})\,,
         \end{equation}
         where $\exp ({\PP_0})$ denotes the image under the exponential map of the set of all $n\times n$ real symmetric matrices with zero trace,
         consisting of all symmetric and positive definite matrices $S$ with $\Det S = 1$, since $ \Det \exp (S) = e^{\Tr(S)}$.
         At the matrix level, the decomposition (\ref{Cartan_group_1}) is nothing but the polar factorization $A = QP$
        of a real nonsingular matrix with unit determinant as the product of a special orthogonal matrix and a symmetric positive definite matrix with unit
        determinant. 

         \begin{rem}\label{Global_Cartan}
         More generally, a global Cartan decomposition of a non-compact, semisimple Lie group  $G$ with Lie algebra $\GG$ is defined as follows.
         There is a unique Lie group automorphism $\Theta$  with differential at the identity given by $\theta$ (the Cartan involution of $\GG$) and such that
         $\Theta^2 = {\rm Id}$ (the identity automorphism). Moreover, there is a compact subgroup $K\subset G$ that is left fixed by $\Theta$.  Let $\PP$ be the
         eigenspace of $\theta$ corresponding to the eigenvalue $-1$.The mapping
         $$K \times \PP \longrightarrow G \quad \text{given \, by} \quad (k,X) \mapsto k \cdot \exp(X)$$
         is a diffeomorphism. The idempotent automorphism $\Theta$ is called a {\rm global Cartan involution}. 
        Writing $P= \exp(\PP)$, we have that the product map $K\times P \rightarrow G$ is a diffeomorphism, so that $G=K\cdot P$.
        If $G=SL(n,\Rr)$, the global Cartan involution is $\Theta: X\mapsto (X^{-1})^T$, hence $K=SO(n,\Rr)$ and $\PP = \PP_0$.
         \end{rem}
         
         We can now determine the additive counterpart of the polar decomposition.
         Taking a smooth path $t\mapsto A(t)$ in $SL(n,\Rr)$ and writing $A(t) = Q(t)P(t)$  (smooth paths in $SO(n,\Rr)$ and $\exp ({\PP_0})$, respectively, with $Q(0) = P(0) = I_n$)
        and differentiating for $t=0$ we find 
        \begin{equation} \label{HSS_0}
        A'(0) = Q'(0) + P'(0)\,,
        \end{equation}
        with $Q'(0)\in \SO_n$ and $P'(0)$ symmetric with $\Tr (P'(0)) = 0$, i.e., $Q'(0)$ is the skew-symmetric part of $A'(0)\in \SL_n$ and $P'(0)$ its symmetric part.
        We can therefore write the Cartan decomposition of $\SL_n$ as the infinitesimal  counterpart of (\ref{Cartan_group_1}), namely:
        \begin{equation}\label{Cartan_alg_0}
        \SL_n = \SO_n \oplus \PP_0\,.
        \end{equation}
        This shows that indeed the symmetric and skew-symmetric splitting of a zero-trace matrix is the linearization near the identity of the polar factorization of matrices with 
        unit determinant.\\
        
        At the Lie group level, the global Cartan decomposition also takes the name of {\textit{KAK decomposition}}: there exist a compact subgroup $\mathcal K$ and an Abelian
        subgroup $\mathcal A$ of $G$ such that 
        \begin{equation} \label{KAK}
	G = \cal{K}\cdot \cal{A} \cdot \cal{K}\,,
	\end{equation}
	so that every $g\in G$ can be written as the product $g = k_1\cdot a \cdot k_2$ with $k_1, k_2 \in \mathcal K$ and $a\in \mathcal A$; see \cite{Edelman2}
	for additional information. 
	When $G= SL(n,\Rr)$, ${\mathcal K} = SO(n,\Rr)$ and ${\mathcal A}$ is the subgroup of diagonal matrices with unit determinant we obtain
	the foremost example of KAK decomposition, namely, the singular value decomposition (SVD):
	$$A = U \Sigma V^T\,,$$
	with $U$ and $V$ real orthogonal $n\times n$ matrices and $\Sigma$ the diagonal matrix with the singular values (in non-increasing order) down
	the main diagonal. Rewriting the SVD as $A = (UV^T)\cdot (V\Sigma V^T)$ we recover the polar factorization $A=QP$ with $Q=UV^T\in SO(n,\Rr)$
	and $P=V\Sigma V^T=(A^TA)^{{\frac12}}$ (symmetric positive definite with $\Det P = 1$). Conversely, starting from $A=QP$ with $Q\in SO(n,\Rr)$ and $P=VDV^T$ positive definite
	(where $V\in SO(n,\Rr)$ and $D$ is diagonal positive definite with $\Det D = 1$) we readily obtain $A=U\Sigma V^T$ with $U=QV\in SO(n,\Rr)$ and $\Sigma = D$.
	It must be noted, however, that the infinitesimal variation of the SVD near the identity does not yield the same tangent space decomposition as the
	linearization of the polar decomposition, and does not lead to an additive splitting of the type considered in this paper.\\
        
        So far we have restricted the discussion to matrices with unit determinant.  Let now $A\in GL(n,\Rr)$. 
       Since any matrix sufficiently close to the identity $I_n$ has positive determinant, we can assume $\Det (A) > 0$. 
        For any such matrix, consider the ``trivial" factorization
        \begin{equation}\label{trivial_1}
      A= \frac{A}{(\Det A)^\frac{1}{n}}\cdot (\Det A)^\frac{1}{n}I_n.
      \end{equation}
     Hence, any matrix close to $I_n$ can be written as the product  of a matrix in $SL(n,\Rr)$ times a matrix of the form $\alpha I_n$, with $\alpha > 0$. 
     At the group level, this corresponds to the decomposition
     \begin{equation}\label{trivial_2}
     GL_+(n,\Rr) = SL(n,\Rr) \cdot \Rr_+^*\,,
     \end{equation}
     where $GL_+(n,\Rr)$ denotes the group of all real $n\times n$ matrices with positive determinant and $\Rr_+^*$ stands for the group of all scalar matrices of the
     form $\alpha I_n$ with $\alpha > 0$. 
         At the Lie algebra level this amounts to the addition of the algebra of \textit{all} scalar matrices to the algebra $\SL_n$, corresponding to the splitting
        $A = A_0 + \Tr (A) I_n$, where $A_0 := A - \Tr (A) I_n\in \SL_n$. That is, we make use again of the Levi decomposition
        \begin{equation} \label{Levi2}
        \GL_n = \SL_n \oplus \RR\,,
        \end{equation}
        where $\RR$ is the subalgebra of all real scalar matrices, including the zero matrix, isomorphic to $\Rr$; note that this is the tangent space, at the identity matrix, to the
        group $\Rr_+^*$ (as well as to the
        group  $\Rr^*$ of invertible scalar matrices). 
        Using (\ref{Cartan_alg_0}) we obtain the  decomposition
        \begin{equation}\label{Cartan_alg_bis}
        \GL_n = \SO_n \oplus \PP_0 \oplus \RR\,,
        \end{equation}
        which is precisely (\ref{Cartan_alg}) since $\PP_0 \oplus \RR= \PP_n$, the space of all $n\times n$ symmetric matrices.\\
        
      %
     
     The Lie algebra decomposition (\ref{Cartan_alg_bis}) is the infinitesimal counterpart of the group (polar) decomposition
     \begin{equation}\label{Cartan_group_tris}
     GL_+(n,\Rr) = SO(n,\Rr)\cdot \exp(\PP_0)\cdot \Rr_+^* = SO(n,\Rr) \cdot \exp (\PP_n),
     \end{equation}
    At the matrix level the decomposition can be written as $A= (\Det A)^{-\frac{1}{n}} \cdot  A \cdot (\Det A)^\frac{1}{n}=
     Q \cdot P$. Here $Q\in SO(n,\Rr)$ and $P=(\Det A)^\frac{1}{n} \cdot P_1$ with $P_1$ such that
   $ (\Det A)^{-\frac{1}{n}} \cdot  A = Q\cdot P_1$ (polar factorization); note that $P_1$ is symmetric positive definite, with $\Det P_1 = 1$ 
   (i.e., $P_1$ is the exponential of a symmetric matrix with zero trace), while $\Det P = \Det A.$\\

         As is well known, the symmetric and skew-symmetric splitting (\ref{HSS}) is the basis for different iterative methods for solving nonsingular linear systems $A x = b$,
	including both stationary methods \cite{Bai,Benzi2004} and non-stationary/preconditioned iterations \cite{Concus,Szyld2022,Widlund}. In Section \ref{sec:4}  we
        take a look at  analogous Lie algebra decompositions of $\GL_n$ and describe the corresponding matrix splittings.\\

\subsection{The Zappa-Sz\'ep decomposition} \label{sec:ZS} 
	
	We begin this subsection by recalling a purely algebraic decomposition of a group as product of two of its subgroups, known as the Zappa-Sz\'ep product,
	see \cite{Szep,Zappa}.\footnote{This decomposition is also know under a few other names, including Zappa-R\'edei-Sz\'ep product, general product, 
	knit product, exact factorization or bicrossed product.}  
	
	\begin{defn}\label{ZS}
	Let $G$ be a group with identity element $e$ and let $H$, $K$ be subgroups of $G$. If $G=HK$ and $H\cap K = \{e\}$, then $G$ is said to be an
	(internal)  {\rm{Zappa-Sz\'ep product}} of $H$ and $K$.  
	\end{defn}
	
	Equivalently, $G$ is the Zappa-Sz\'ep product of $H$ and $K$ if for every $g\in G$ there exist unique $h\in H$ and $k\in K$ such that $g=hk$.  We also observe that
	if $G=HK$ is a Zappa-Sz\'ep product, then also $G=KH$ is.    We stress that in this product neither $H$ nor $K$ need to be normal subgroups of $G$, hence this
	product is more general than the semidirect product of subgroups. Note that the (global) Cartan decomposition of $SL(n,\Rr)$ (or $GL(n,\Rr)$)
	 is not a Zappa-Sz\'ep product since the factor $\exp ({\PP}$) is 
	not a subgroup of $SL(n,\Rr)$ or $GL(n,\Rr)$.\\
	
	In the case of (matrix) Lie groups, the Zappa-Sz\'ep decomposition induces a splitting of the corresponding Lie algebra as a direct sum
	of subalgebras,  as shown by the following result. Although straightforward, we haven't been able to find this result in the literature and thus 
	we include a proof.\\

	\begin{prop} \label{ZS_splitting}
	Let the Lie group $G$ be the Zappa-Sz\'ep product of Lie subgroups $H\subset G$ and $K\subset G$, and let $\GG$, $\HH$ and $\GKK$ be the corresponding Lie algebras. Then
	$$ \GG = \HH \oplus \GKK\,.$$
	\end{prop}
	
	\proof
	First we note that $\HH$ and $\GKK$ are Lie subalgebras of $\GG$ and therefore $\HH + \GKK \subseteq \GG$.  To show that equality holds and that $\HH \cap \GKK =\{0\}$
	we use a dimensionality argument. The map $(h,k) \mapsto g = hk$ is evidently a diffeomorphism of $H\times K$ onto $G$ (as differentiable manifolds), implying that 
	$$\dim (G) = \dim(H\times K) = \dim(H) + \dim(K)\,.$$
	It follows that the corresponding Lie algebras must also satisfy
	$$\dim(\GG) = \dim(\HH) + \dim(\GKK) $$
	and therefore $\HH \cap \GKK =\{0\}$, showing that $ \GG = \HH \oplus \GKK\,.$
	\endproof

	\begin{rem}
	It is worth mentioning that while $\HH$ and $\GKK$ are Lie subalgebras of $\GG$, the direct sum in Proposition \ref{ZS_splitting} need not be
	a direct sum of Lie algebras, unless the condition $[\HH,\GKK]=0$ holds. In general, it is only a direct sum of vector subspaces. An example of a 
	direct sum decomposition of a Lie algebra is the Levi decomposition $\GL_n = \SL_n \oplus \RR$, since $[\SL_n, \RR] = 0$. In order to distinguish between
	a direct sum decomposition of a Lie algebra and 
	the decomposition of a Lie algebra as a direct sum of vector subspaces, which may or may not be Lie subalgebras,  the symbol $\oplus_L$ is sometimes 
	used instead of $\oplus$.
	\end{rem}

	An example of Zappa-Sz\'ep product is given by the QR factorization \cite{GVL4}. Any nonsingular matrix $A$ can be factorized as $A=QR$ with $Q$ orthogonal
	and $R$ upper triangular with positive diagonal entries, and the factors are unique.  Letting $G=GL(n, \Rr)$, $H=O(n,\Rr)$ and $K=UT_+(n,\Rr)$, 
	the group of upper triangular matrices with positive diagonal entries,
	we have that $G=HK$,
	a Zappa-Sz\'ep product since $H\cap K = \{I_n\}$.
	Moreover, in this case $G$ is a Lie group and $H$, $K$ are Lie subgroups.   The corresponding
	Lie algebras are, respectively, $\GL_n$, $\SO_n$ and $\UT_n$ (the algebra of all $n\times n$ upper triangular matrices). The foregoing Proposition \ref{ZS_splitting} implies that
	
	\begin{equation}\label{skew-tri}
	\GL_n = \SO_n \oplus \UT_n\,.
	\end{equation}

	Equation (\ref{skew-tri}) simply expresses the fact that \textit{any real square matrix $A$ can be written in one and only one way as the sum of a 
	skew-symmetric matrix and an upper triangular matrix}. 
	This \textit{skew-symmetric and upper triangular splitting} is the basis for various iterative methods
	and preconditioners for solving nonsymmetric  linear systems, such as the STS and MSTS iterations (see, e.g., \cite{Bai_et_al,Krukier2002,Wang}).  The splitting takes the form
	\begin{equation}\label{SUT-split}
	A = (L_0 - L_0^T)  + (D + U_0 + L_0^T) 
	\end{equation}  
	where $A = L_0 + D + U_0$ is the usual splitting of a matrix into its strictly lower triangular, diagonal, and strictly upper triangular parts. 
	Of course, one may also consider  instead the \textit{skew-symmetric and lower triangular splitting},
	\begin{equation}\label{SLT-split}
	A = (U_0 - U_0^T)  + (D + L_0 + U_0^T) \,,
	\end{equation}  
	corresponding to the direct sum decomposition
	\begin{equation}\label{skew-lowertri}
	\GL_n =   \SO_n \oplus \LT_n\,,
	\end{equation}
	which in turns arises from linearization of the Zappa-Sz\'ep decomposition 
	$$GL(n,\Rr) = LT_+ (n, \Rr) \cdot O(n, \Rr),$$
	 which is  sometimes called the LQ factorization ($A=LQ$ with $L$ lower triangular with positive diagonal entries and $Q$ orthogonal).

       \begin{rem}
       Similar to the case of the Cartan decomposition, we can also derive (\ref{skew-tri}) without using Proposition \ref{ZS_splitting}, as follows. We first deal with the
       case of a matrix with unit determinant. 
       Consider a smooth matrix path $t\mapsto A(t) = Q_+(t) \cdot R_1(t)$ in $SL(n,\Rr)$ with $t\in \Rr$ in an interval around zero, $Q_+(t)$ a path in $SO(n,\Rr)$ and
       $R_1(t)$ a path in $SUT_n(n,\Rr)$ (the subgroup of upper triangular matrices with unit determinant), with $Q_+(0) = R_1(0) = I_n$.  Differentiating for $t=0$
       we obtain
       $$A'(0) = Q_+'(0) R_1(0) + Q_+(0) R_1'(0) = Q_+'(0) + R_1'(0) \,.$$
       This shows that any matrix in the Lie algebra $\SL_n$ can be written, uniquely, as the sum of a skew-symmetric matrix, $Q_+'(0) \in \SO_n$, and an upper triangular
       matrix with zero trace, $R_1'(0)\in \SUT_0$, hence
       \begin{equation}\label{skew-tri_0}
       \SL_n = \SO_n \oplus \SUT_0\,.
       \end{equation}
       To obtain the general case we make use again of the Levi decomposition (\ref{Levi}) and observe that
       $$\GL_n = \SL_n \oplus \RR = \SO_n \oplus \SUT_0 \oplus \RR\,,$$
       where $\RR$ is the Lie ideal of $\GL_n$ consisting of all the scalar matrices.  Observing that
       $$ \SUT_0 \oplus \RR = \UT_n\,,$$
       the Lie algebra of all $n\times n$ upper triangular matrices, 
       we obtain the direct sum decomposition (\ref{skew-tri}).  This is the infinitesimal counterpart of the (global) QR factorization of a nonsingular matrix.
       \end{rem}
       
       \begin{rem}
       It is worth noting that the vector subspaces $\UT_n$ and $\LT_n$, besides being Lie algebras with respect to the commutator product, are also Jordan algebras
       with respect to the Jordan product. This is not surprising in view of the fact that every associative algebra is automatically a Lie algebra under the commutator
       product and a Jordan algebra under the Jordan product \cite{McCrimmon}, and both $\UT_n$ and $\LT_n$ are associative algebras under the standard matrix product. 
       We explicitly note, however,  that the direct sum decompositions (\ref{skew-tri}) and (\ref{skew-lowertri}) of $\GL_n$ are not orthogonal with respect to the Frobenius
       inner product.  
       \end{rem}
       
       
       A further, simple example of Zappa-Sz\'ep decomposition is the decomposition of the group $GL_+(n,\Rr)$ given by (\ref{trivial_1})-(\ref{trivial_2}).
      As already pointed out, linearizing this decomposition near the identity $I_n$ yields the Lie algebra splitting
      $$\GL_n = \SL_n \oplus \RR\,,$$
      which is precisely the Levi decomposition (\ref{Levi2}) of $\GL_n$, corresponding to the matrix splitting $A=(A-\Tr(A) I_n) + \Tr (A) I_n$.
       
	
	\subsection{The Iwasawa decomposition} \label{sec:Iwasawa}
	
	The QR factorization also occurs as a special case of the \textit{Iwasawa decomposition}; see  for example \cite{Helgason,Knapp}, also referred to as
	the `KAN' decomposition. 
	If $G$ is a connected semisimple Lie group, there exists a compact subgroup ${\cal K}\subset G$, an Abelian subgroup ${\cal A}\subset G$, and a 
	nilpotent subgroup ${\cal N}\subset G$ such that 
	
	\begin{equation} \label{Iwasawa}
	G = \cal{K}\cdot \cal{A} \cdot \cal{N}\,.
	\end{equation}
	
	Consider the case where $G =SL(n,\Rr)$. Let  ${\cal K}=SO(n,\Rr)$, let 
	$\cal{A}$ be the subgroup of all diagonal matrices with positive diagonal entries and unit determinant, and ${\cal N}=UT_1(n,\Rr)$ (the subgroup of all 
	$n\times n$ unit triangular matrices). These are all connected Lie subgroups of $SL(n,\Rr)$ with $\cal{K}$ compact, $\cal{A}$ Abelian,
	and $\cal{N}$ nilpotent.  We can take a smooth path $t\mapsto A(t)$ in $SL(n,\Rr)$
	passing through the origin of the form 
	\begin{equation}\label{QDR}
	A(t) = Q_+(t)\cdot D_+(t)\cdot U(t)\,,
	\end{equation}
		with $Q_+(t)\in \cal{K}$, $D_+(t)\in \cal{A}$, $U(t) \in \cal N$ and 
		$Q_+(0) = D_+(0) = U(0) = I_n$. This is just the QR factorization written in QDR form, where the `D' factor is diagonal and the `R' factor
		is  unit upper triangular. 
         Differentiating (\ref{QDR}) at $t=0$ we obtain the splitting
	$$A'(0) = Q_+'(0) + D_+'(0) + U'(0)$$
	with $Q_+'(0) \in \SO_n$, $D_+'(0)$ a diagonal matrix, and $U'(0)$ strictly upper triangular. 
	Hence, the Iwasawa decomposition of the group $SL(n,\Rr)$ is just the QDR factorization, and its linearization near the identity yields the 
	Iwasawa decomposition of the semisimple Lie algebra $\SL_n$:
	\begin{equation}\label{skew-tri_Iwasawa}
	\SL_n = \SO_n \oplus \DD_0 \oplus \UT_0\,,
	\end{equation}
	where $\DD_0$ denotes the subalgebra of diagonal matrices with zero trace and $\UT_0$ that of strictly upper triangular matrices.
         Noting that $\DD_0 \oplus \UT_0 = \SUT_0$, the Lie algebra of all upper triangular matrices with zero trace, we recover the splitting (\ref{skew-tri_0}).\\
         
         For the general case ($G=GL(n,\Rr)$, with Lie algebra $\GL_n$) we follow the usual script; using the Levi decomposition (\ref{Levi})  we obtain the direct sum decomposition
         \begin{equation}\label{skew-tri_Iwasawa_gen}
         \GL_n = \SL_n \oplus \RR = \SO_n \oplus \DD_0 \oplus \UT_0 \oplus \RR = \SO_n \oplus \UT_n\,,
         \end{equation}
         i.e., the splitting of an arbitrary square matrix as the sum of a skew-symmetric matrix and an upper triangular matrix, both uniquely determined: if $A = L_0+D+U_0$,
         then $A = (L_0-L_0^T) + (D + U_0 + L_0^T)$, exactly as in subsection \ref{sec:ZS}.  This splitting is the basis for the method first studied in \cite{Bai_et_al}.
         We will come back to this method in Section \ref{subsec:STS}.

	\subsection{The LU decomposition and the Gauss-Seidel and Jacobi splittings}\label{sec:LDU}
	Perhaps the best known factorization of a square matrix is the triangular factorization, or LU decomposition. 
	It is a basic fact that if all the leading principal minors of $A\in GL(n,\Rr)$ are nonzero, there exist a unit lower triangular matrix $L$
	and an upper triangular matrix $U$ such that $A=LU$, and this decomposition is unique. 
	This is sometimes called the \textit{Doolittle factorization} of $A$; in alternative, we can write $A=LU$ with $L$ lower
	triangular and $U$ unit upper triangular (this is known as the \textit{Crout factorization}).  \\

         It is sometimes convenient to express the
	factorization in LDU form:
	\begin{equation}\label{LDU}
	A  = L\,D\,U\,,
	\end{equation}
	where $L$ and $U$ are unit triangular 
	and $D$ is the diagonal matrix of pivots. Then $A=L\cdot (DU)$ is the Doolittle factorization, and $A = (LD)\cdot U$ is Crout.   \\
	
	It is also well known that if $A\in GL(n,\Rr)$ does not have the LU factorization, there exists a permutation matrix $P$ such that $PA$ has
	the LU (or LDU) factorization or, equivalently, 
	\begin{equation}\label{PLDU}
	A = P^T\,L\,U \,.
	\end{equation}
	However, this factorization is no longer unique in general. It is also well known that when implementing the LU factorization on a computer
	a permutation matrix $P$, representing row interchanges (partial pivoting), is typically incorporated as a means to promote numerical stability 
	even though the matrix $A$ may have the LU factorization. In this paper, however, we
	 leave finite precision and numerical stability issues aside.\\
	
	While the LU (or LDU) factorization is a staple of numerical linear algebra, it is much less used in the context of Lie theory.  In its place, the
	\textit{Bruhat decomposition} plays an important role.  The Bruhat decomposition of a nonsingular matrix has the form
	
	\begin{equation}\label{Bruhat}
	A = U_1\,  P\,  U_2\,,
	\end{equation}
	with $U_1$, $U_2$ invertible upper triangular matrices and $P$ a permutation matrix.\footnote{Some authors define the Bruhat decomposition using lower
	triangular instead of upper triangular factors. In practice, this makes no difference.} At the group level, the Bruhat decomposition is
	\begin{equation}\label{Global_Bruhat}
	GL(n,\Rr) = UT(n,\Rr)\cdot  S_n\cdot  UT(n,\Rr)\,,
	\end{equation}
	where $UT(n,\Rr)$ is the group of all invertible upper triangular matrices and $S_n$ denotes the group of all $n\times n$ permutation matrices
	(isomorphic to the symmetric group on $n$ objects).	A variant is the \textit{modified Bruhat decomposition}, defined in \cite{Tyrt}:
	$$A = L\, \Pi \, U\,,$$
	with $L$ lower triangular, $\Pi$ a permutation matrix, and $U$ upper triangular.\footnote{We point out that the 
	modified Bruhat decomposition is sometines referred to as the \textit{Gelfand-Naimark decomposition}, see \cite{Tam}; however, no reference to Gelfand
	or Naimark is given in \cite{Tam}.}\\
	
	Unfortunately, none of these triangular factorizations is a Zappa-Sz\'ep decomposition of $GL(n,\Rr)$. Depending on the decomposition, either existence or uniqueness
	of the decomposition is not satisfied. Specifically, existence may fail for the LU factorization, and uniqueness for its permuted versions. 
	The set of invertible matrices for which the LDU factorization exists (and is unique), however, is dense in $GL(n,\Rr)$ and moreover,
	the LDU factorization exists and is unique for all matrices $A\in GL(n,\Rr)$ that are sufficiently close to the identity $I_n$; in other words,
	there exists $\varepsilon > 0$ such that  if $A\in B(I_n, \varepsilon)$, the open ball of radius $\varepsilon$ centered at $I_n$, then $A=LU$ (or $A=LDU$)
	exists and is unique.\footnote{Here the topology is the one defined by the operator norm induced by the vector $2$-norm, but other matrix norms could be used as well, since 
	all norms on $\Rr^{n\times n}$ induce the same topology.}  
	Indeed, the determinant is a continuous function, hence a matrix sufficiently close to $I_n$ will have all its leading principal minors different
	from zero, guaranteeing the existence and uniqueness of the LU (or LDU) factorization.   Therefore, we can linearize the LDU factorization near
	the identity matrix to determine its infinitesimal counterpart.\\

	To this end, consider for sufficiently small $|t|$ a smooth path of 
	invertible matrices $t\mapsto A(t)\in B(I_n, \varepsilon)$ of the
	form 
	\begin{equation}\label{path_lu} 
	A(t) = L(t)U(t)\qquad {\text{with}} \quad L(0) = I_n, \,\,\, U(0) = I_n \,,
	\end{equation}
	where $L(t)\in LT_1(n,\Rr)$ and  $U(t)\in UT(n,\Rr)$. Differentiating (\ref{path_ldu}) at $t=0$ we obtain
	\begin{equation}\label{linearized_lu}
	A'(0) = L'(0) + U'(0)
	\end{equation}
	with $L'(0)$ strictly lower triangular and $U'(0)$ upper triangular.  At the Lie algebra level we have the direct sum decomposition
	\begin{equation}\label{Lie_Doolittle}
	\GL_n = \LT_0 \oplus \UT_n\,.
	\end{equation}

	 Letting $A'(0) = A$, $L'(0) = L_0$ and $U'(0) = U = D + U_0$
	with $D$ diagonal and $U_0$ strictly upper triangular, we see that the lower and upper triangular splitting
	\begin{equation}\label{tri_split}
	A = L_0 + U
	\end{equation}
	is the linearization of the Doolittle form of the LU factorization.  In the case of the Crout form we have again the smooth factorization
	(\ref{path_lu}) but now with $L(t)\in LT(n,\Rr)$ (lower triangular) and $U(t)\in UT_1(n,\Rr)$ (unit upper triangular). In this case, linearization
	near the identity yields the direct sum decomposition
	\begin{equation}\label{Lie_Crout}
	\GL_n = \LT_n \oplus \UT_0\,.
	\end{equation}
  In matrix terms this corresponds to the splitting
	\begin{equation}\label{tri_split_GS}
	A = L + U_0\,,
	\end{equation}
	where $L = L_0 + D$.\\
	
	We observe that both the Lie algebra splittings (\ref{Lie_Doolittle}) and (\ref{Lie_Crout}) are orthogonal direct sums with respect to the Frobenius
	inner product, since
	$$ \LT_0 = (\UT_n)^\perp \qquad {\text{and}} \qquad \UT_0 = (\LT_n)^\perp$$
	as vector subspaces of $\Rr^{n\times n}$.\\

	In numerical linear algebra, when $D$ is nonsingular the splitting (\ref{tri_split_GS}) corresponds to the Gauss-Seidel method, whereas (\ref{tri_split}) corresponds to
	the Gauss-Seidel method applied in reverse order.  It is an interesting coincidence that linearization of the LU factorization underlying Gaussian elimination leads to the
	Gauss-Seidel splittings. If $t\mapsto A(t) = A(t)^T$, then for $|t|$ small enough $A(t)$ will be positive definite, the LDU factorization becomes the root-free Cholesky
	factorization $A=LDL^T$, and linearizing at $t=0$ yields
	the symmetric matrix splitting $A'(0)=L_0 + D_0 + L_0^T$.\\

		We can also write 
	\begin{equation}\label{path_ldu} 
	A(t) = L(t) D(t) U(t) 
	\end{equation}
	with $D(t)$ diagonal such that $D(0)=I_n$, $L(t)$ unit lower triangular and $U(t)$ unit upper triangular, and obtain (\ref{tri_split}) or (\ref{tri_split_GS}) by differentiating
	(\ref{path_ldu}) at $t=0$.  The Lie algebra splitting can be then written as the orthogonal direct sum
	
	\begin{equation}\label{ldu_split}
	 \GL_n = \LT_0  \oplus \DD \oplus \UT_0\,.
	 \end{equation}\\	 
	 Writing this as $\GL_n = (\LT_0  \oplus \DD) \oplus \UT_0$ or as $\GL_n = \LT_0  \oplus (\DD \oplus \UT_0)$ leads to the two types of Gauss-Seidel splittings.\\
	
	We note incidentally that the decomposition of the matrix group $UT(n,\Rr)$ as the product of the subgroup
	of (invertible) diagonal matrices and that of the unit upper triangular ones and the decomposition of the matrix group $LT(n,\Rr)$ as
	the product of the subgroup of the unit lower triangular matrices and that of (invertible) diagonal ones are both examples of Zappa-Sz\'ep products. \\
	
	The splitting associated to the Jacobi stationary iteration can be written as
	\begin{equation}\label{jac}
	A = D_0 + (L_0 + U_0) = D_0 - (A-D_0)\,.
	\end{equation}
	
	If we introduce the $n(n-1)$-dimensional vector subspace
	\begin{equation}\label{W}
	 \WW_n = \{ C=(c_{ij})\in \Rr^{n\times n} \, | \,  c_{ii} = 0 \,\,\, \forall i=1,\dots ,n \}
	 \end{equation}
	consisting of all the real $n\times n$ matrices with zero diagonal entries we have that
	$$\WW_n = \LT_0  \oplus \UT_0$$
	and therefore the Jacobi splitting corresponds to the orthogonal direct sum
	\begin{equation}\label{jacobi_split}
	 \GL_n =  \DD \oplus \WW_n\,.
	 \end{equation}
	
         Note that $\WW_n$ is not a Lie algebra with respect to
	the commutator product.  It is the tangent space at the identity of the matrix manifold 
	$$W_n = LT_1(n,\Rr)\cdot UT_1(n,\Rr)\,,$$
	product of the Lie subgroups of $SL(n,\Rr)$ consisting of all unit lower and unit upper triangular matrices. 
	Note hat $W_n$ is not a group, consistent with the fact that $\WW_n$ is not an algebra. \\

	\subsection{Spectral decompositions}  	
	In linear algebra there are a number of canonical forms of matrices that can be cast as multiplicative matrix decompositions, such as the Jordan form $A=X(D + N)X^{-1}$
	(with $D$ diagonal and $N$ a particular nilpotent matrix such that $DN=ND$) and the Schur form $A = URU^T$ with $U$ unitary and $R$ upper triangular. 
	These reduce to $A=XDX^{-1}$ and $A = UDU^T$  in the diagonalizable and normal case, respectively.
	We refer to these as `spectral decompositions.' \\
	
	We will not consider infinitesimal counterparts of these decompositions in this paper. The main reason is that, generically, a matrix $A\in GL(n,\Rr)$ has complex
	eigenvalues, and thus in order to interpret these spectral decompositions as factorizations of a group into subgroups (even if only locally) would imply working
	in the complex group $GL(n,\Cc)$ and its Lie algebra, which we do not consider in this paper.  Restricting the analysis to classes of matrices with real eigenvalues
	does not solve the problem, since these matrix classes do not, in general, have a Lie group structure, and moreover linearizing near the
	identity cannot result in a global direct sum decomposition of the tangent space of $GL(n,\Rr)$ (or $SL(n,\Rr))$ into subspaces, since any neighborhood of the 
	identity matrix $I_n$ in $GL(n,\Rr)$ contains matrices with non-real eigenvalues. \\

	A possibility would be to replace the complex Schur form of $A$ with its real Schur form. The real Schur decomposition is $A=QRQ^T$  with $Q\in O(n,\Rr)$ and $R\in \Rr^{n\times n}$ 
	block upper triangular with $2\times 2$ diagonal blocks corresponding to complex conjugate eigenvalues and $1\times 1$ blocks corresponding to the real eigenvalues (when present).
	Doing this requires extending our treatment from point to block factorizations and block splittings; this can also be done, in principle, for the LU and QR factorization.  We leave these
	generalizations for future work. 
		
	\section{Lie-Jordan splittings}  \label{sec:4} 
	
	As already mentioned, the splitting of a matrix into its symmetric and skew-symmetric parts is the basis for various iterative algorithms
	and preconditioning techniques. A prominent example of this is  the well-known \textit{Hermitian and skew-Hermitian splitting} (HSS) method:
	\begin{equation}\label{HSS_iteration}
		\begin{cases}
			(\alpha I_n + \mathcal{H}(A) ) x^{(k+\frac{1}{2})}=(\alpha I_n -\mathcal{S}(A) ) x^{(k)}+b\\
			(\alpha I_n +\mathcal{S}(A) ) x^{(k+1)} = (\alpha I_n - \mathcal{H}(A) ) x^{(k+\frac{1}{2})}+b
		\end{cases}\quad 	(k=0,1,2,\ldots),
	\end{equation}
	first proposed  in \cite{Bai}, where it is shown that if $\mathcal{H}(A)$ is positive definite, the HSS method converges to the unique solution of 
	the linear system $Ax=b$, for any initial guess and  $\alpha >0$. In \cite{Benzi2004}, the HSS method  was extended to generalized saddle point problems in which the Hermitian part of the coefficient     matrix of the  system  is possibly singular, and the use of HSS as a preconditioner was also proposed and investigated. Since then, hundreds of papers have ben written on this method and its many variants;   see, e.g., the treatment in \cite{BaiBook} and references therein.\\
	
	As we saw in subsection \ref{sec:Cartan}, the splitting $A={\mathcal S}(A) + {\mathcal H}(A)$ corresponding to the direct sum decomposition (\ref{Cartan_alg})
	is a splitting of the total matrix space $\Rr^{n\times n}$ as the orthogonal sum of a Lie algebra and a Jordan algebra. This is not the only example of such a splitting; as we saw in
	subsection \ref{sec:LDU}, the Gauss-Seidel splittings (\ref{Lie_Doolittle}) and (\ref{Lie_Crout}) are also of this type. We call such a splitting a \textit{Lie-Jordan splitting}, 
	see Definition \ref{def:LJS} below. The skew-symmetric and triangular 
	splitting and the Jacobi splitting, on the other hand, are not Lie-Jordan splittings. \\	
	\begin{defn}\label{def:LJS}
	We say that a matrix $A\in \GL_n=\Rr^{n\times n}$ admits a (non-trivial) {\rm Lie-Jordan splitting} if it can be written as $A = B+C$ (with $B\ne 0$, $C\ne 0$) where $B\in \GG$, $C\in \Jj$, 
	with $\GG$ a Lie subalgebra and $\Jj$ a Jordan subalgebra such that $\Rr^{n\times n} = \GG \oplus \Jj$ (direct sum of vector subspaces) and  $\GG \perp \Jj$ with respect to the Frobenius
	inner product.	\\
	\end{defn}
	
	 In the next
	two subsections we will describe two additional examples of Lie-Jordan splittings.  Although not under this name, these Lie algebra splittings (and the 
	corresponding group-level decompositions) are well known in Lie theory;  a recent treatment can be found in Section 3 of \cite{Edelman2}, see also \cite[Ch.~2]{Gilmore}.
  In Section \ref{sec:LJ}  we show how these splittings lead to generalizations of the HSS  iterative method for solving linear systems.\\

	A key feature of the symmetric and skew-symmetric 
	direct sum decomposition of $\Rr^{n\times n}$ is that the Lie algebra component $\SO_n$
	(the skew-symmetric matrices)
	is the tangent space at the identity of the orthogonal group $O(n,\Rr)$, which is the group of linear transformations on $\Rr^n$ that preserve the 
	standard (Euclidean) inner product.  The Jordan
	algebra component (the space of symmetric matrices) is the orthogonal complement of the tangent space to $O(n,\Rr)$ with respect to the Frobenius 
	inner product $\langle A,B\rangle_F= \Tr (A^TB)$. In other words, it is the normal space to $O(n,\Rr)$ at the identity.\\
	
	A natural question, then, is to ask what happens if the standard inner product in $\Rr^n$, $\langle x,y\rangle = x^Ty$, is replaced with a different bilinear form, associated with a matrix
	$J\in \Rr^{n\times n}$: 
	$$b(x,y) = x^TJy$$ 
	(which reduces to the standard inner product when $J=I_n$).   Notice that, in principle, $J$ need not be symmetric positive definite; that is,
	$J$ may not define an inner product.  Recalling now that the classical groups $O(p,q,\Rr)$ and $Sp(n,\Rr)$, like $O(n,\Rr)$,  can all be defined as the
	sets of matrices $Q$ such that $b(Qx,Qy)=b(x,y)$ for all $x, y\in \Rr^n$ for a suitable bilinear form $b$,
	it is natural to begin by examining the cases of the pseudo-orthogonal group $O(p,q, \Rr)$ and the symplectic
	group $Sp(n,\Rr)$.  We do this in the next two subsections.  \\
	
	For future reference, we recall that the {\textit{$J$-adjoint}} of a matrix $A\in \Rr^{n\times n}$ is defined as $A^\star = J^{-1} A^T J$, and that a matrix is 
	{\textit{$J$-symmetric}} if $A=A^\star$ and
	{\textit{$J$-skew-symmetric}} if $A = - A^\star$.

	\subsection{The pseudo-symmetric and pseudo-skew-symmetric splitting} \label{pseudo-HSS}
	
	Let $p$ and $q$ be positive integers and let  $n=p+q$.  Consider the bilinear form associated with the matrix
	\begin{equation}\label{Ipq}
	 J = I_{p,q} = \left[ {\begin{array}{*{20}{c}}
				I_p&0\\
				0& -I_q
		\end{array}}\right ] \,.
	\end{equation}
	The group leaving $I_{p,q}$ invariant is the \textit{pseudo-orthogonal} (or \textit{indefinite}) group $O(p,q,\Rr)$. It can also be characterized as the group of all matrices $A$
	satisfying  the equivalent conditions $I_{p,q}A^TI_{p,q}= A^{-1}$, 
	$A^T I_{p,q} A = I_{p,q}$, or $A= I_{p,q} (A^T)^{-1} I_{p,q}$.\\
	
	 The Lie algebra (tangent space at the identity) of the group $O(p,q,\Rr)$ 
	consists of the \textit{$(p,q)$-pseudo-skew-symmetric matrices},
	which have the following structure (see, e.g., \cite{AMM}):
	$$ A = \left[ {\begin{array}{*{20}{c}}
				F &G^T \\
				G & H
		\end{array}}\right ] \,, \quad F= - F^T\in \Rr^{p\times p},\,\,\, H = -H^T \in \Rr^{q\times q}, \,\,\, G\in \Rr^{q\times p}\,.
		$$
	Note that this is precisely the set of all matrices $A\in \Rr^{n\times n}$ which satisfy the condition
	$$A^T I_{p,q} + I_{p,q}A = 0$$
	(a linearized version of the condition $A^T I_{p,q} A = I_{p,q}$).
	This set of matrices is clearly a vector subspace of $\Rr^{n\times n}$, is closed under the commutator product,
	and it is  a Lie subalgebra of dimension $\frac{p(p-1)}{2} + \frac{q(q-1)}{2} + pq = \frac{n(n-1)}{2}$ of $\GL_n$. We shall denote this Lie algebra
	by $\SO_{p,q}$.  Next, consider the set of $n\times n$ matrices of the form
			$$ A = \left[ {\begin{array}{*{20}{c}}
				F &G^T \\
				-G  & H
		\end{array}}\right ] \,, \quad F= F^T\in \Rr^{p\times p},\,\,\, H = H^T \in \Rr^{q\times q}, \,\,\, G\in \Rr^{q\times p}\,.
		$$
		This is the set of \textit{$(p,q)$-pseudo-symmetric} matrices.
		It is readily verified that the matrices in this set are precisely those that satisfy the condition 
		$$A^T I_{p,q} - I_{p,q}A = 0\,.$$
		This set of matrices is again a vector subspace of $\Rr^{n\times n}$, is closed under the Jordan product,
		and it is a Jordan algebra, which we denote by $\PP_{p,q}$. Its dimension is $\frac{p(p+1)}{2} + \frac{q(q+1)}{2} + pq = \frac{n(n+1)}{2}$,
		hence as a subspace of $\GL_n=\Rr^{n\times n}$ it is complementary to $\SO_{p,q}$:
		\begin{equation}\label{PS}
		\GL_n=  \Rr ^{n\times n} = \SO_{p,q} \oplus \PP_{p,q}\,.
		\end{equation}
		
		The decomposition (\ref{PS}) states that any matrix ${\cal A}\in \Rr^{(p+q)\times (p+q)}$ can be written in one and only one way as the
		sum of  a matrix $H\in \PP_{p,q}$ and a matrix $ S \in \SO_{p,q}$. The splitting is given by 
				\begin{equation}\label{Pseudo_HSS}	
                {\mathcal A} = \left[ {\begin{array}{*{20}{c}}
				A & B \\
				C  & D
		\end{array}}\right ] = \left[ {\begin{array}{*{20}{c}}
				\frac{1}{2}(A+A^T)  & \frac{1}{2}(B-C^T)\vspace{0.1in}  \\
				-\frac{1}{2}(B^T-C)  & \frac{1}{2}(D+D^T)
		\end{array}}\right ]  +	
		\left[ {\begin{array}{*{20}{c}}
				\frac{1}{2}(A-A^T)  & \frac{1}{2}(B+C^T) \vspace{0.1in} \\
				\frac{1}{2}(B^T+C)  & \frac{1}{2}(D-D^T)
		\end{array}}\right ]  = H+S\,.
	       \end{equation}
		
		We call the splitting (\ref{Pseudo_HSS}) the \textit{pseudo-symmetric and pseudo-skew-symmetric splitting}.  
		It is easily checked that  $H\in \PP_{p,q}$, $S\in \SO_{p,q}$, and  $\Tr(H^TS) = 0$. Indeed, (\ref{PS}) is an orthogonal
		direct sum decomposition, and (\ref{Pseudo_HSS}) is an example of a Lie-Jordan splitting. It depends, of course, on the choice of $p$ and $q$
		(where $p+q=n$). In the case $p=n$ and $q=0$,  we have $O(p,q,\Rr)$ = $O(n,\Rr)$, ${\cal A} = A \in \Rr^{n\times n}$, and
		(\ref{PS})-(\ref{Pseudo_HSS}) reduce to the usual symmetric and skew-symmetric splitting (\ref{Cartan_alg})-(\ref{HSS}).\\
		
		In Section \ref{sec:Cartan} we saw how the symmetric and skew-symmetric splitting can be interpreted as the infinitesimal version of the polar factorization. 
		It is thus natural to ask
		if the pseudo-symmetric and pseudo-skew-symmetric splitting corresponds to a kind of polar factorization.  In terms of Lie algebras and Lie groups, noting that 	
		 (\ref{PS}) is a 
		 decomposition of $\GL_n$ into eigenspaces corresponding to the $\pm 1$ eigenvalues of the
		   involution $\theta (X) = - I_{p,q}A^TI_{p,q}$, it is natural to ask if there exists an associated
		 global 
		 decomposition of $GL(n,\Rr)$, of which (\ref{PS}) is the infinitesimal counterpart. This \textit{indefinite polar factorization}, if it exists, must take the form
		 \begin{equation}\label{polar_indef}
		 A=QP \,\, \mbox{ where } \,\, A\in GL(n,\Rr),\,\, Q\in O(p,q,\Rr),\,\, P \in \exp (\PP_{p,q})\,.
		 \end{equation}
		 The idempotent diffeomorphism $\Theta$ leaving $O(p,q,\Rr)$ fixed is given by $\Theta(X) = I_{p,q} (X^T)^{-1} I_{p,q}$.\\
		 
		 As it turns out, the decomposition (\ref{polar_indef}) exists under certain conditions on $A$.  These conditions have been established in \cite{Higham05},
		 where the \textit{generalized polar decomposition} is investigated; see also \cite{Tisseur06}. In particular, specialization of Theorem 4.1 in \cite{Higham05} to the pseudo-orthogonal group
		 implies that $A\in GL(n,\Rr)$ has the factorization (\ref{polar_indef}) if and only if the pseudo-symmetric matrix 
		 $A^\star A =I_{p,q}A^TI_{p,q}A$ has no eigenvalues on the negative real 
		 axis.\footnote{Note that if either $p=0$ or $q=0$, then $A^\star A =I_{p,q}A^TI_{p,q}A=A^TA$, which is positive definite and therefore satisfies the condition for the existence of the
		 (standard) polar factorization, for any $A$.}
		 Hence, for a ``generic" $A\in GL(n,\Rr)$ the generalized polar decomposition (\ref{polar_indef}) exists,
		 in the sense that the set of matrices for which the decomposition exists is dense in $GL(n,\Rr)$. 
		 Moreover, it is easily seen that as $A$ approaches the identity matrix
		 in $GL(n,\Rr)$,
		 $I_{p,q}A^TI_{p,q}A$ also tends to the identity and therefore has no eigenvalues on the negative real axis. In other words, there exists an open neighborhood of the
		 identity matrix in $GL(n,\Rr)$ such that every matrix in this neighborhood has the generalized polar decomposition (\ref{polar_indef}).  We can therefore linearize this
		 factorization near the identity to obtain the splitting (\ref{Pseudo_HSS}). The situation is similar to the one we encountered when discussing the LDU factorization, which 
		 always exists in a sufficiently small neighborhood of the identity.
		 
		\begin{rem}
		It is natural to introduce a ``pseudo-HSS iteration", i.e., an  iterative method of the form (\ref{HSS_iteration}) but based on the
		 pseudo-symmetric and pseudo-skew-symmetric splitting (\ref{Pseudo_HSS}), and to seek convergence conditions for it. We do this
		 in Section \ref{sec:LJ}.
		 \end{rem}

	\subsection{The Hamiltonian and skew-Hamiltonian splitting}  \label{HSH}
	
	Let $n=2m$ (even) and consider the anti-symmetric bilinear form associated with the matrix
	\begin{equation}\label{sym}
	J =  \Sigma_n= \left[ {\begin{array}{*{20}{c}}
				0 & I_m\\
				-I_m& 0
		\end{array}}\right ] \,.
	\end{equation}
	The group leaving $\Sigma_n$ invariant is the {\textit{symplectic group}} $Sp(n,\Rr)$. These are the matrices $A$ that satisfy the equivalent conditions 
	$\Sigma_n A^T \Sigma_n= -A^{-1}$,  $A^T \Sigma_n A = \Sigma_n$, or $A=- \Sigma_n(A^T)^{-1} \Sigma_n $.\\

        The Lie algebra of the group $Sp(n,\Rr)$ 
	consists of the \textit{Hamiltonian matrices},
	which have the following structure:
	$$ A = \left[ {\begin{array}{*{20}{c}}
				F &G \\
				H & -F^T
		\end{array}}\right ] \,, \quad F\in \Rr^{m\times m},\,\,\, G= G^T\in \Rr^{m\times m},\,\,\, H = H^T \in \Rr^{m\times m}\,,
		$$
	see \cite{AMM}. Note that this is precisely the set of all matrices $A\in \Rr^{n\times n}$ which satisfy the condition
	$$A^T \Sigma_n + \Sigma_n A = 0$$
	(a linearized version of the condition $A^T \Sigma_n A = \Sigma_n$).
	This set of matrices is clearly a vector subspace of $\Rr^{n\times n}$, is closed under the commutator product,
	and it is  a Lie subalgebra of dimension $m(m+1)+ m^2 = 2m^2 + m$ of $\GL_n$. We shall denote this Lie algebra
	by $\SP_n$.  Next, consider the set of $n\times n$ matrices of the form
	$$ A = \left[ {\begin{array}{*{20}{c}}
				F &G \\
				H & F^T
		\end{array}}\right ] \,, \quad F\in \Rr^{m\times m},\,\,\, G= -G^T\in \Rr^{m\times m},\,\,\, H = -H^T \in \Rr^{m\times m}\,.
		$$
		A matrix of this type is said to be \textit{skew-Hamiltonian}.
                It is readily verified that the skew-Hamiltonian matrices are precisely those that satisfy the condition 
		$$A^T \Sigma_n - \Sigma_n A = 0\,.$$
		This set is again a vector subspace of $\Rr^{n\times n}$, and it forms a Jordan algebra with respect to the Jordan product.
		We denote it by $\JJ_n$. Its dimension is $m(m-1)+ m^2 = 2m^2 - m$,
		hence as a subspace of $\GL_n=\Rr^{n\times n}$ it is complementary to $\SP_n$:
		\begin{equation}\label{PH}
		\GL_n=  \Rr ^{n\times n} = \SP_n \oplus \JJ_n\,.
		\end{equation}

The decomposition (\ref{PH}) states that any matrix ${\cal A}\in \Rr^{n\times n}$ with $n=2m$ can be written in one and only one way as the
		sum of  a matrix  $S \in \SP_n$ and a matrix $H\in \JJ_n$. The splitting is given by 
		\begin{equation}\label{Ham_HSS}	
		{\mathcal A} = \left[ {\begin{array}{*{20}{c}}
				A & B \\
				C  & D
		\end{array}}\right ] = \left[ {\begin{array}{*{20}{c}}		
		\frac{1}{2}(A-D^T)  & \frac{1}{2}(B+B^T) \vspace{0.1in} \\ 			
				\frac{1}{2}(C+C^T)  & -\frac{1}{2}(A^T-D)
		\end{array}}\right ]   +		
		 \left[ {\begin{array}{*{20}{c}}
		\frac{1}{2}(A+D^T)  & \frac{1}{2}(B-B^T)\vspace{0.1in}  \\ 				
				\frac{1}{2}(C-C^T)  & \frac{1}{2}(A^T+D)
		\end{array}}\right ]   =  S+ H \,.
	       \end{equation}

We call the splitting (\ref{Ham_HSS}) the \textit{Hamiltonian and skew-Hamiltonian splitting}.  It is only defined for $n$ even.
		It is easily checked that $ S\in \SP_n$ and $H\in \JJ_n$,  and that  $\Tr(S^TH) = 0$. Indeed, (\ref{PH}) is an orthogonal
		direct sum decomposition, and (\ref{Ham_HSS}) is another example of a Lie-Jordan splitting. \\
		
		We note that (\ref{PH}) is a 
		decomposition of $\GL_n$ into eigenspaces corresponding to the eigenvalues $\pm 1$ of the involution $\theta (X) =  \Sigma_n X^T\Sigma_n$,
                 with the idempotent diffeomorphism $\Theta$ leaving $Sp(n,\Rr)$ fixed being given by $\Theta(X) = -\Sigma_n (X^T)^{-1} \Sigma_n$.
		As in the previous subsection, specializing Theorem 4.1 in \cite{Higham05} to the symplectic group implies that a matrix $A\in GL(n,\Rr)$ ($n$ even) 
		has a ``symplectic" polar factorization of the form
		\begin{equation}\label{polar_symp}
		A = WZ  \,\, \mbox{ where } \,\, W\in Sp(n,\Rr),\,\, Z \in \exp (\JJ_n)
		 \end{equation}
		 if and only if the skew-Hamiltonian matrix $-\Sigma_n A^T \Sigma_n A$ has no eigenvalues on the negative real axis.  Again, this is a condition
		 that holds generically. It is also easy to see that as $A$ approaches the identity matrix in $GL(n,\Rr)$, so does $-\Sigma_n A^T \Sigma_n A$, hence
		 there is a neighborhood of the identity in $GL(n,\Rr)$ such that every matrix in this neighborhood has the symplectic polar factorization (\ref{polar_symp}).
		 Therefore, the splitting (\ref{Ham_HSS}) is  obtained as the linearization near the identity of the factorization
		 (\ref{polar_symp}).
		 
		 \begin{rem}
		 As in the case of the pseudo-symmetric and pseudo-skew-symmetric splitting, 
		 we can introduce an HSS-like iteration, i.e., an  iterative method of the form (\ref{HSS_iteration}) but based on the
		 Hamiltonian and skew-Hamiltonian splitting (\ref{Ham_HSS}), and to seek convergence conditions for it. We do this
		 in Section \ref{sec:LJ}.
                  \end{rem}

		\subsection{Quadratic algebra decomposition} \label{sec:quadratic}

		 The symmetric and skew-symmetric, pseudo-symmetric and pseudo-skew-symmetric, and the Hamiltonian and skew-Hamiltonian splittings are all
		 instances of what we called Lie-Jordan splittings; the first one corresponds to the Cartan decomposition of $\GL_n$, the other two to
		 decompositions that are not Cartan but only 
		 ``Cartan-like," all well known in Lie theory. Moreover,  their global variants are all types of polar factorizations which are known in
		 numerical linear algebra.
		 In this subsection we discuss one possible unified treatment of these splittings, showing that they are all special cases of a general construction. To this end we recall
	         the definition of a $J$-quadratic Lie algebra and  the corresponding notion of a $J$-quadratic Jordan algebra, slightly more general than the definition in \cite{Bloch}.\\
	         
	         Here we use the notation $\GL_n$ for the space of all $n\times n$ matrices $\Rr^{n\times n}$ endowed with the Lie algebra structure
	         given by the commutator product, and the notation $\Mm(n,\Rr)$ for the same space endowed with the Jordan algebra structure given by the
	         Jordan product.\\
	         
        
Let $J$ be an invertible matrix in $\Rr^{n\times n}$ such that $J=J^T$ or and $J=-J^T$.
	         
\begin{defn}
Given $J$ as above, a Lie subalgebra $\GG$ of $\GL_n$ is called a {\rm $J$-quadratic Lie algebra} if 
		\begin{equation}\label{eq:quadratic_alg1}
			W\in \GG \iff W^T J + J W = 0. 
		\end{equation}
\end{defn}

Related to a $J$-quadratic Lie algebra, we have the notion of $J$-quadratic Jordan algebra.

\begin{defn}
Given $J$ as above, a Jordan subalgebra $\Jj$ of $\Mm(n,\Rr)$ is called a {\rm $J$-quadratic Jordan algebra} if 
		\begin{equation}\label{eq:quadratic_alg1}
			W\in \Jj\iff W^T J - J W = 0.
		\end{equation}
\end{defn}

\begin{rem}
We observe that up to a change of coordinates, any $J$-quadratic Lie algebra in the hypothesis above is isomorphic to either $\SO_n,\SO_{p,q}$ or $\SP_n$. 
\end{rem}

\section{The $J$-HSS iteration}\label{sec:LJ}

In this Section we obtain a convergence result for a generalization of the HSS algorithm for solving linear systems which we call the $J$-HSS iteration;
it can be described as an \textit{alternating Lie-Jordan iteration}. 


\subsection{A general convergence result}


Let us consider the linear system of equations in $\Rr^n$:
\begin{equation}\label{eq:lin_syst}
Ax=b,
\end{equation}
where $A\in GL(n,\Rr)$.
For any $J$ invertible symmetric or skew-symmetric, we can define an involution on $\Rr^{n\times n}$ as $A\mapsto -JA^TJ^{-1}$.
This map is an isometry with respect to the positive definite form $\langle A,B\rangle = \Tr(A^TJ^TJB)$.
By Theorem 3.14 in \cite{Artin},  we have the following splitting:
\begin{equation}\label{eq:splitting}
\Rr^{n\times n} = \GG \oplus \Jj\,,
\end{equation}
where $\GG$ is the $J$-quadratic Lie algebra, corresponding to the eigenspace of the involution with eigenvalue $1$, and $\Jj$ the $J$-quadratic Jordan algebra, corresponding to the eigenspace of the involution with eigenvalue $-1$.
More explicitly, we can define  the $J$-symmetric part of $A$ by $H=\frac{1}{2}(A+JA^TJ^{-1})\in\Jj$ and  its $J$-skew-symmetric part by $S=\frac{1}{2}(A-JA^TJ^{-1})\in\GG$.\\

Consider the following alternating (two-step) scheme, which we refer to as the $J$-HSS iteration:
\begin{equation}\label{eq:HSS_alg}
\begin{cases}
(H+\alpha J^{-1})x^{k+\frac{1}{2}}&=(\alpha J^{-1}-S)x^{k}+b\\
(S + \alpha J^{-1})x^{k+1}&=(\alpha J^{-1}-H)x^{k+\frac{1}{2}}+b
\end{cases}  \quad 	(k=0,1,2,\ldots).
\end{equation}\\

\begin{prop}\label{convergence}
Let $A, J\in \Rr^{n\times n}$ be invertible. Write
$A= H+S$ where $H\in \Jj$ and $S\in \GG$, where $\Jj$ and $\GG$ are $J$-quadratic Jordan and Lie algebras (resp.) such that  (\ref{eq:splitting}) holds.
\begin{enumerate}
\item[(i)] If $J=J^T$, and the symmetric matrix $HJ$ is positive definite, then the iteration (\ref{eq:HSS_alg}) is well-defined and converges to the unique solution of $Ax=b$
for all $\alpha > 0$ and $x^0\in \Rr^n$.
\item[(ii)] If $J=-J^T$, and the symmetric matrix $SJ$ is positive definite, then the iteration (\ref{eq:HSS_alg}) is well-defined and converges to the unique solution of $Ax=b$
for all $\alpha > 0$ and $x^0\in \Rr^n$. In this case, $n$ is necessarily even.
\end{enumerate}
\end{prop}

\begin{proof}
The linear system $Ax = b$ can be rewritten as 
$$AJy = b, \quad  x=Jy\,.$$
Then
\begin{equation}\label{eq:hss_split}
AJ=(H+S)J = HJ + SJ\,.
\end{equation}
If $J=J^T$, then  $HJ =  \frac{1}{2}(AJ+JA^T)$ is
symmetric and $SJ=\frac{1}{2}(AJ-JA^T) $ is skew-symmetric, hence the
splitting (\ref{eq:hss_split}) is just the standard HSS splitting of $AJ$. Moreover, 
$H+\alpha J^{-1} = (HJ + \alpha I)J^{-1}$, hence $H+\alpha J^{-1}$ is invertible (as the product of invertible matrices)
when $HJ$ is positive definite; moreover, $S+\alpha J^{-1}=(SJ+\alpha I)J^{-1}$ is invertible for all $\alpha\ne 0$, since $SJ$
is skew-symmetric, and the iteration is well-defined. By Theorem 2.2 in \cite{Bai}, the sequence $y^k = J^{-1}x^k$ with
$x^k$ defined by iteration (\ref{eq:HSS_alg}) converges, for all $\alpha > 0$ and $y^0=J^{-1}x^0$, to the unique solution $y^*$ of 
$AJy = b$, hence $x^k$ converges to the unique solution $x^*$ of $Ax=b$.  This proves part (i). \\

To prove (ii), note that when $J=-J^T$ then $HJ$ and $SJ$ in (\ref{eq:hss_split}) are, respectively, skew-symmetric and symmetric.
The remainder of the proof is analogous to the one of part (i). Note that in this case $n$ must be even, since $J$ must be nonsingular.
\end{proof}

It is instructive to also prove Proposition \ref{convergence} directly, without relying on Theorem 2.2 in \cite{Bai}.
Let $J=J^T$ and assume $HJ$ is positive definite. Eliminating $x^{k+\frac{1}{2}}$ from  (\ref{eq:HSS_alg}) we can rewrite the $J$-HSS iteration as 
$x^{k+1}=T_\alpha x^k + c$, where
\[T_\alpha = (\alpha I+JS)^{-1}(\alpha I -JH)(JH+\alpha I)^{-1}(\alpha I-JS)
=(\alpha I+SJ)^{-1}(\alpha I -HJ)(HJ+\alpha I)^{-1}(\alpha I-SJ)\]
and 
\begin{equation}\label{eq:c}
c = 2\alpha (\alpha I+JS)^{-1}(JH+\alpha I)^{-1}Jb = 2\alpha J (SJ+\alpha I)^{-1} (HJ+\alpha I)^{-1}b.
\end{equation}
We notice that $T_\alpha$ is similar to 
\[\widehat{T}_\alpha = (\alpha I-HJ)(HJ+\alpha I)^{-1}(\alpha I-SJ)(\alpha I+SJ)^{-1},\]
hence the spectral radius of $T_\alpha$ is equal to that of $\widehat{T}_\alpha$.
Since $(SJ-\alpha I)(SJ+\alpha I)^{-1}$ is orthogonal as the Cayley transform of a skew-symmetric matrix, we get
\[
\varrho(T_\alpha) = \varrho (\widehat T_\alpha) 
\le \|\widehat T_\alpha\|_2 \le \|(\alpha I-HJ)(\alpha I+HJ)^{-1})\|_2 = \max_{1\leq i\leq n}\dfrac{|\alpha-\lambda_i(HJ)|}{|\alpha+\lambda_i(HJ)|} < 1,\]
where $\lambda_i(HJ)$ denotes the $i$-th eigenvalue of the symmetric matrix $HJ$. Hence the iteration converges, for all $\alpha > 0$ and all $x^0$.
Moreover, the upper bound $ \max_{1\leq i\leq n}\dfrac{|\alpha-\lambda_i(HJ)|}{|\alpha+\lambda_i(HJ)|}$ is minimized by
$$\alpha_* = \sqrt{\lambda_{\min} (HJ) \lambda_{\max} (HJ)}.$$
The case $J=-J^T$ is analogous, interchanging the roles of $HJ$ and $SJ$. \\

\begin{rem}
When $J=I_n$, the $J$-HSS iteration reduces to the standard HSS iteration  (\ref{HSS_iteration}) introduced in \cite{Bai}.
Also note that the $J$-HSS iteration is mathematically equivalent to the standard HSS iteration applied to the linear system $AJy=b$,
$y=J^{-1}x$, when the splitting $A=H+S$ is a Lie-Jordan splitting with $S\in \GG$ and $H\in \Jj$, where $\GG$ and $\Jj$ are,
respectively, $J$-quadratic Lie and Jordan algebras. \\
\end{rem}

\begin{rem}
The first half-step of the $J$-HSS iteration requires solving linear systems with coefficient matrix $H + \alpha J^{-1}$.   Writing $H+\alpha J^{-1} = (HJ+\alpha I)J^{-1}$,
this reduces to solving a system with symmetric positive definite coefficient matrix $HJ+\alpha I$ followed by multiplication by $J$.  Since $J$ is usually of very simple form,
the cost of the first half-step is essentially identical to that for the standard HSS method.   Likewise, the second half-step requires solving systems with coefficient matrix
$SJ + \alpha I$, where $SJ$ is skew-symmetric, followed by a matrix-vector product with $J$, again essentially identical to the standard HSS iteration in terms of cost.
 \\
\end{rem}

\begin{rem}
The single-step iteration $x^{k+1} = T_\alpha x^k + c$ can also be written as 
$$x^{k+1} = (I - P_\alpha^{-1}A)x^k + P_\alpha^{-1}b$$
with 
\begin{equation}\label{eq:prec}
P_\alpha^{-1} =  2\alpha J (SJ+\alpha I)^{-1} (HJ+\alpha I)^{-1}b,  \quad \mbox{or} \quad 
P_\alpha = \frac{1}{2\alpha} (HJ+\alpha I)(SJ +\alpha I)J^{-1}\,,
\end{equation}
see for instance \cite{BS97}.
We refer to the matrix $P_\alpha$ as the {\rm{$J$-HSS preconditioner}}. Its application requires the solution of two
shifted linear systems, one of which is symmetric positive definite and the other shifted skew-symmetric, and a matrix-vector product with $J$. The
scalar factor $\frac{1}{2\alpha}$ is immaterial for preconditioning, and it can be ignored.
Since the matrix $J$ is usually of very simple form, the
cost is essentially the same as for the standard HSS preconditioner.  When used to accelerate the convergence of
Krylov subspace methods, exact solves are usually replaced with inexact ones (see, e.g., \cite{Bai2, Benzi2004}). \\
\end{rem}

\begin{rem}
We mention that in \cite{Greif2022}, the author has proposed two variants of the HSS scheme that bear some resemblance to the $J$-HSS
method, but should not be confused with it.  The first variant uses the (standard) HSS splitting $A=H+S$ and considers the iteration obtained
by replacing shifts of the form $\alpha I$ with $\alpha J$ where $J=\Sigma_n$, with $n$ even.  In the second variant the symmetric
matrix $H$  is shifted by $\alpha I$ in the first half step, while the skew-symmetric matrix $S$ is shifted by $\alpha J$ in the second half step
(here, again, $J=\Sigma_n$). For these iterations to be well-defined, restrictions must be imposed on the parameter $\alpha$ or on the matrix $S$;
moreover, convergence conditions appear to be difficult to establish, see \cite{Greif2022}.
\end{rem}


In the next two subsections we discuss the two choices of the matrix $J$ from subsections \ref{pseudo-HSS} and \ref{HSH},
which provide ``canonical' examples of the two cases covered in Proposition \ref{convergence}.

\subsection{Pseudo-HSS method.}\label{sec:phss}

When $J= I_{p,q}$ (see \ref{Ipq}), we refer to the $J$-HSS iteration as the \textit{pseudo-HSS} iteration. 
Here we assume that $p\ne0$, $q\ne 0$, for otherwise we revert to the standard HSS iteration.
As shown in subsection \ref{pseudo-HSS}, every ${\cal A}\in \Rr^{n\times n}$ (with $n=p+q$) can be split as
$$
  {\mathcal A} = \left[ {\begin{array}{*{20}{c}}
				A & B \\
				C  & D
		\end{array}}\right ] = \left[ {\begin{array}{*{20}{c}}
		\frac{1}{2}(A+A^T)  & \frac{1}{2}(B-C^T)\vspace{0.1in}  \\
				-\frac{1}{2}(B^T-C)  & \frac{1}{2}(D+D^T)
		\end{array}}\right ] +
		 \left[ {\begin{array}{*{20}{c}}		
				\frac{1}{2}(A-A^T)  & \frac{1}{2}(B+C^T) \vspace{0.1in} \\
				\frac{1}{2}(B^T+C)  & \frac{1}{2}(D-D^T)
		\end{array}}\right ]   =  H+ S\,.
		$$
		
		We have
		
		\begin{equation}\label{HJ_a}
		HJ =\left[ {\begin{array}{*{20}{c}}
		\frac{1}{2}(A+A^T)  & -\frac{1}{2}(B-C^T)\vspace{0.1in}  \\
				-\frac{1}{2}(B^T-C)  & -\frac{1}{2}(D+D^T)
		\end{array}}\right ] ,
		\end{equation}
		a symmetric matrix. For this matrix to be positive definite, it is necessary that both $\frac{1}{2}(A+A^T) $ and $ -\frac{1}{2}(D+D^T)$ be positive definite. 
A necessary and sufficient condition is that both $\frac{1}{2}(A+A^T) $ and the Schur complement matrix

\begin{equation}\label{Schur_a}
-\frac{1}{2}(D+D^T) - \frac{1}{2}(B^T-C)(A+A^T) ^{-1}( B-C)^T
\end{equation}
be positive definite.  In general, this condition is not easy to check.  An exception is when $C=B^T$, in which case 
$$HJ =\left[ {\begin{array}{*{20}{c}}
		\frac{1}{2}(A+A^T)  & 0 \vspace{0.1in}  \\
                                0             & -\frac{1}{2}(D+D^T)
		\end{array}}\right ] ,
	 $$
	  a block diagonal matrix, which is positive definite when both $A$ and $-D$ are positive real. In this case we also have
	$$SJ =\left[ {\begin{array}{*{20}{c}}
		\frac{1}{2}(A-A^T)  & -B \vspace{0.1in}  \\
                                B^T             & -\frac{1}{2}(D-D^T)
		\end{array}}\right ] .
	 $$
	 
	 Hence, when $B=C^T$ he two half-steps of the pseudo-HSS iteration require solving shifted linear systems with coefficient matrices given by
	 
	 \begin{equation}\label{PHSS_matrix1}
	 HJ+\alpha I = \left[ {\begin{array}{*{20}{c}}
		\frac{1}{2}(A+A^T) + \alpha I  & 0 \vspace{0.1in}  \\
                                0             & \alpha I -\frac{1}{2}(D+D^T)
		\end{array}}\right ] ,
		\end{equation}
		and
	\begin{equation} \label{PHSS_matrix2}
	SJ+\alpha I = \left[ {\begin{array}{*{20}{c}}
		\frac{1}{2}(A-A^T) +\alpha I  & -B \vspace{0.1in}  \\
                                B^T             & \alpha I -\frac{1}{2}(D-D^T)
		\end{array}}\right ] .
          \end{equation}
          
          If in addtition $D=D^T$, as is frequently the case in the solution of saddle point problems (see \cite{BGL05}), we have
		\begin{equation}\label{PHSS_matrix1b}
	 HJ+\alpha I = \left[ {\begin{array}{*{20}{c}}
		\frac{1}{2}(A+A^T) + \alpha I  & 0 \vspace{0.1in}  \\
                                0             & \alpha I - D
		\end{array}}\right ] ,
		\end{equation}
		a symmetric positive definite, block diagonal matrix for all $\alpha > 0$, and
	\begin{equation} \label{PHSS_matrix2b}
	SJ+\alpha I = \left[ {\begin{array}{*{20}{c}}
		\frac{1}{2}(A-A^T) +\alpha I  & -B \vspace{0.1in}  \\
                                B^T             & \alpha I 
		\end{array}}\right ] .
          \end{equation}
          In this special case, the resulting iterative method is essentially the one introduced in \cite{Benzi2004}.
          Systems with coefficient matrix given by (\ref{PHSS_matrix1b}) can be solved by any solver 
          for symmetric positive definite systems.  The shifted skew-symmetric systems with coefficient matrix
          (\ref{PHSS_matrix2b}) can be reduced, via Schur complement reduction, to solving $p\times p$ systems
          with matrix $\frac{1}{2}(A-A^T)  + \alpha I + \frac{1}{\alpha}BB^T$, for which several specialized solvers exist;
          see, for instance, \cite{Szyld2022} and \cite{BF23}. \\
          
          We note that while the partitioning of $n$ as $n=p+q$ is in principle arbitrary, in practice the choice of $p$ and $q$
          may be induced by a natural partitioning of the variables, as in saddle point problems; for example, in fluid flow problems
          $p$ could be the number of velocity degrees of freedom and $q$ the number of pressure degrees of freedom.  In other
          cases the choice of $p$ and $q$ may be dictated by computational considerations.

	 \subsection{Hamiltonian-HSS method.}\label{sec:hhss}
	 When $n$ is even and $J=\Sigma_n$, we refer to the $J$-HSS iteration as the \textit{Hamiltonian-HSS} iteration. 
As shown in subsection \ref{HSH}, every ${\cal A}\in \Rr^{n\times n}$ can be split as
$$		{\mathcal A} = \left[ {\begin{array}{*{20}{c}}
				A & B \\
				C  & D
		\end{array}}\right ] = 
		\left[ {\begin{array}{*{20}{c}}		
		\frac{1}{2}(A-D^T)  & \frac{1}{2}(B+B^T) \vspace{0.1in} \\ 			
				\frac{1}{2}(C+C^T)  & -\frac{1}{2}(A^T-D)
		\end{array}}\right ]   +		
		\left[ {\begin{array}{*{20}{c}}
		\frac{1}{2}(A+D^T)  & \frac{1}{2}(B-B^T)\vspace{0.1in}  \\ 				
				\frac{1}{2}(C-C^T)  & \frac{1}{2}(A^T+D)
		\end{array}}\right ]  =  S+H\,,
$$
with $S\in \SP_n$ (Hamiltonian) and  $H\in \JJ_n$ (skew-Hamiltonian).\\

We have

$$HJ = \left[ {\begin{array}{*{20}{c}}
		-\frac{1}{2}(B-B^T)  & \frac{1}{2}(A+D^T)\vspace{0.1in}  \\ 				
			-\frac{1}{2}(A^T+D)  &  \frac{1}{2}(C-C^T)
		\end{array}}\right ],
$$
a skew-symmetric matrix, and

$$ SJ =  \left[ {\begin{array}{*{20}{c}}
		-\frac{1}{2}(B+B^T)  & \frac{1}{2}(A-D^T)\vspace{0.1in}  \\ 				
			\frac{1}{2}(A^T-D)  &  \frac{1}{2}(C+C^T)
		\end{array}}\right ],
$$
a symmetric matrix, which must be positive definite for the Hamiltonian-HSS method to converge for all $\alpha > 0$.
In particular, a necessary condition is that $B$ and $C$ have, respectively, negative definite and positive definite symmetric part.
In the special case in which $B=B^T$, $C=C^T$ and $A=D^T$ these matrices reduce to
$$HJ = \left[ {\begin{array}{*{20}{c}}
		0  & A\vspace{0.1in}  \\ 				
                -A   &  0
		\end{array}}\right ],
$$
and

$$ SJ =  \left[ {\begin{array}{*{20}{c}}
		-B  & 0 \vspace{0.1in}  \\ 				
	          0 &  C
		\end{array}}\right ],
$$
and the Hamiltonian-HSS scheme reduces to a special case of the standard HSS method, which converges for all $\alpha > 0$ if 
$B$ is negative and $C$ is positive definite.  When $C=-B$ and $A=D$ we have
$${\cal A} =  \left[ {\begin{array}{*{20}{c}}
				A & B \\
				-B  & A
		\end{array}}\right ] ,
		$$
which is one of the real forms of the complex matrix $C=A+i B$. In this case we have
$${\cal A} J=  \left[ {\begin{array}{*{20}{c}}
				-B & A \\
				-A  & -B
		\end{array}}\right ] = HJ + SJ,
		$$
where
$$HJ = \left[ {\begin{array}{*{20}{c}}
		-\frac{1}{2}(B-B^T)  & \frac{1}{2}(A+A^T)\vspace{0.1in}  \\ 				
			-\frac{1}{2}(A+A^T)  &- \frac{1}{2}(B-B^T)
		\end{array}}\right ],
$$
a skew-symmetric Hamiltonian matrix, and

$$ SJ =  \left[ {\begin{array}{*{20}{c}}
		-\frac{1}{2}(B+B^T)  & \frac{1}{2}(A-A^T)\vspace{0.1in}  \\ 				
			-\frac{1}{2}(A-A^T)  & - \frac{1}{2}(B+B^T)
		\end{array}}\right ]
$$
a symmetric skew-Hamiltonian matrix. If in addition $A=A^T$, $SJ$ becomes block diagonal
and the Hamiltonian-HSS iteration converges if $-B$ is positive real, a condition often satisfied 
in applications; see, e.g., \cite{BBC10,BBCW,BB08}.

\section{Further discussion of alternating iterations}\label{sec:6}

In this Section we give additional details on the  triangular and skew-symmetric splitting iteration (following
\cite{Bai_et_al}, see also \cite{BaiBook}) and we discuss a further class of splitting methods obtained by
linearizing the Kronecker product decomposition. The latter class includes the well-known Alternating Direction
Implicit (ADI) method as a special case.

\subsection{Convergence of the triangular and skew-symmetric splitting iteration} \label{subsec:STS}

Similar to the HSS method, the authors of \cite{Bai_et_al} have proposed an alternating iteration based
on the triangular and skew-symmetric splitting of $A$. If $A= L_0 + D + U_0$ is the usual splitting of $A$
into its strictly lower triangular, diagonal and strictly upper triangular parts (or their block analogues), we
can write $A = U + S $  with $U = U_0 + D + L_0^T$ (upper triangular) and $S = L_0 - L_0^T$ (skew-symmetric).   
One can then define the
following alternating iteration:
\begin{equation}\label{STS_iteration}
		\begin{cases}
			(\alpha I_n + U ) x^{(k+\frac{1}{2})}=(\alpha I_n - S) x^{(k)}+b\\
			(\alpha I_n + S ) x^{(k+1)} = (\alpha I_n - U) x^{(k+\frac{1}{2})}+b
		\end{cases}\quad (k=0,1,2,\ldots),
	\end{equation}
	where $\alpha > 0$.  Note hat the first subsystem in (\ref{STS_iteration}) is triangular, and therefore easy to solve, while the coefficient
	matrix subsystem is a shifted skew-symmetric matrix, as in the HSS method.
	The convergence analysis of this method is very similar to that of the HSS iteration.
	In particular, the method is unconditionally convergent if $A$ is positive definite.
	Indeed, the iteration matrix $T_\alpha$ of this algorithm is similar to
	$$\hat T_\alpha =  (\alpha I_n - U)(\alpha I_n + U)^{-1} (\alpha I_n - S)(\alpha I_n + S)^{-1}$$
	so that for all $\alpha > 0$ we have
	$$\rho (T_\alpha) = \rho(\hat T_\alpha) \le \|(\alpha I_n - U)(\alpha I_n + U)^{-1}\|_2 \| (\alpha I_n - S)(\alpha I_n + S)^{-1}\|_2 < 1.$$
	The contractivity of $(\alpha I_n - U)(\alpha I_n + U)^{-1}$ follows from the fact that $U$ is positive definite, since 
	$$\frac{1}{2}(U + U^T )= \frac{1}{2}(U_0 + D + L_0^T + U_0^T + D +L_0 )= \frac{1}{2}(A+A^T).$$
	Kellogg's Lemma (see, e.g, \cite{Marchuk}) insures that  $\|(\alpha I_n - U)(\alpha I_n + U)^{-1}\|_2 < 1$; see also \cite{Bai_et_al}.\\
	
	As already observed in Sections \ref{sec:ZS}-\ref{sec:Iwasawa}, the underlying Lie algebra decomposition
	$$ \GL_n = \SO_n \oplus \UT_n = \SO_n \oplus \DD_0 \oplus \UT_0  $$
	 is the linearization of the QR (or QDR) factorization of the
	general linear group $GL(n, \Rr)$, a special case of the Zappa--Sz\'ep (or of the Iwasawa) decomposition.
	Unlike the Lie algebra decompositions considered in the previous Section, this is {\em not} a quadratic algebra
	decomposition.
	
	\subsection{ADI-like splittings} \label{sec:ADI}
	
	Let now $A, B\in \Rr^{n\times n}$ and consider the $n^2\times n^2$ matrix $M = A\otimes I_n + I_n\otimes B$, the 
	{\em Kronecker sum} of $A$ and $B$. The assumption that $A$ and $B$ are of the same size is not necessary, and is made here only for notational simplicity.
	To solve the linear system $M x = b$, let $\alpha > 0$ and consider the alternating
	iteration
	\begin{equation}\label{ADI_iteration}
		\begin{cases}
			(\alpha I_{n^2} + A\otimes I_n ) x^{(k+\frac{1}{2})}=(\alpha I_{n^2} - I_n\otimes B) x^{(k)}+b\\
			(\alpha I_{n^2} + I_n\otimes B ) x^{(k+1)} = (\alpha I_{n^2} - A\otimes I_n) x^{(k+\frac{1}{2})}+b
		\end{cases}\quad (k=0,1,2,\ldots).
	\end{equation}

	Note that the second of the two subystems in (\ref{ADI_iteration}) is block diagonal (with all diagonal blocks equal to $B$), 
	while the first one can be symmetrically permuted to a system of the same type. 
	The classical Alternating Direction Implicit (ADI) method, see for instance \cite{Marchuk,PR,Varga}, is of this form.  When applied to the standard
	five-point difference approximation of the Poisson equation on a rectangular region (with an equal number of mesh points in each direction),
	$A = B$ is just the tridiagonal approximation of the 1D Laplacian and $M$ is the discrete  2D Laplacian. As is well-known, for this problem the ADI method is
	unconditionally convergent.\\
	
	More generally, it is straightforward to check that iteration (\ref{ADI_iteration}) is unconditionally convergent to the solution of $M x = b$ if $A$ and $B$
	are positive semidefinite (not necessarily symmetric) with at least one of them positive definite, in which case $M$ is also necessarily positive definite.
	Indeed, the iteration matrix of (\ref{ADI_iteration}) is similar to 
	$$ (\alpha I_{n^2} - A\otimes I_n)(\alpha I_{n^2} + A\otimes I_n)^{-1} (\alpha I_{n^2} - I_n\otimes B )(\alpha I_{n^2} + I_n\otimes B)^{-1}$$
	and under the stated assumptions we have, for all $\alpha >0 $, that
	$$\|(\alpha I_{n^2} - A\otimes I_n)(\alpha I_{n^2} + A\otimes I_n)^{-1}\|2 \le 1 \quad {\rm and} \quad
	\|(\alpha I_{n^2} - I_n\otimes B )(\alpha I_{n^2} + I_n\otimes B)^{-1}\|_2 \le 1,$$
	with at least one of the two norms being $< 1$, since $A\otimes I_n$ and $I_n \otimes B$ inherit the positive (semi)definiteness properties of $A$ and $B$.\\
	
	The alternating iteration (\ref{ADI_iteration}) is obtained from the natural splitting
	\begin{equation}\label{ADI_split}
	M = (A\otimes I_n) + (I_n\otimes B),
	\end{equation}
	and it is natural to ask whether this splitting corresponds to a linearization of a Lie group decomposition. Not surprisingly, the answer is affirmative and the
	corresponding group factorization is the Kronecker (or tensor) product decomposition. This decomposition differs significantly from the ones previously considered
	in this paper since it is not a global decomposition of the group $GL(n^2, \Rr)$ but only of a particular subgroup of it, namely, 
	\begin{equation} \label{Kron_prod}
	G = GL(n,\Rr)\otimes GL(n,\Rr) = \{ A\otimes B\ |\,  A, B \in GL(n,\Rr)\}.
	\end{equation}
	This is to be expected, since not every $n^2\times n^2$ real matrix is of the form (\ref{ADI_split}). 
	The fact that $G$ is a subgroup of $GL(n^2,\Rr)$ follows immediately from the properties of the Kronecker (or tensor) product of matrices.  The identity element
	is $I_{n^2} = I_n \otimes I_n$ and the inverse of $A\otimes B$ is $A^{-1}\otimes B^{-1}$. Moreover, $(A\otimes B)\cdot (C\otimes D) = (AB)\otimes (CD)$.
	The dimension of $G$ as a Lie group is $2n^2 - 1$.   To see this, observe that the group $G$ can also be written as
	\begin{equation} \label{Kron_gr}
	G = GL(n,\Rr) \otimes GL(n,\Rr)/\Rr^*
	\end{equation}
	Indeed, any matrix
	of the form $A\otimes B$ with $A\in GL(n,\Rr)$ and $B\in GL(n,\Rr)$ can be written equivalently as $A_1 \otimes B_1$ with  $A_1\in GL(n,\Rr)$ and $B_1$ a matrix
	of determinant $\pm 1$, owing to the identity
	$$ A\otimes B = \left (t A\right )\otimes  \left (t^{-1} B \right ), \quad  t\in \Rr^*\,.$$
	Note that the ordering chosen for the factors is arbitrary (we could have equally well picked $A$ with $\Det A = \pm 1$ and $B\in GL(n,\Rr)$). 
	Since $GL(n,\Rr)$ has dimension $n^2$ and $GL(n,\Rr)/\Rr^*$ has dimension $n^2-1$, $G$  has dimension $2n^2 - 1$.\\

Next, we consider the linearization of the Kronecker product near the identity. 
Let $t\mapsto A(t)$ and $t \mapsto B(t)$ be smooth paths in $GL(n,\Rr)$ and $SL(n,\Rr)$, respectively, such that 
$A(0) = I_n = B(0)$. Then clearly $t\mapsto M(t) = A(t) \otimes B(t)$ is a smooth path in $G$, such that $M(0) = I_{n^2}$, and differentiating at $t=0$ leads to
$$ M'(0) = A'(0) \otimes I_n + I_n\otimes B'(0), $$
with $A'(0) \in \GL_n$ and $B'(0)\in \SL_n$. Hence, the tangent space of the Lie group $G$ at the identity $I_{n^2}$ is precisely the Lie algebra of all 
matrices that are Kronecker sums of two matrices:
\begin{equation}\label{Kron_split}
\GG = \{ M \in \GL_{n^2}\, | \, M = A\otimes I_n + I_n \otimes B \,, A \in \GL_n, \, B\in \SL_n\},
\end{equation}
the dimension of which (as a vector space) is clearly $2n^2 -1$, as it should be.\footnote{This Lie algebra is identical to the Lie algebra of all matrices
of the form $M = A\otimes I_n + I_n \otimes B \,, A \in \GL_n, \, B\in \GL_n$, as the trivial identity
$$ A\otimes I_n + I_n \otimes B = (A + \Tr(B) I_n) \otimes I_n + I_n\otimes  (B - \Tr(B) I_n)$$
shows.} We note in passing that the fact that $\GG$ is the Lie algebra of the Lie group $G$ also follows easily from the well-known identity
$$ \exp(A\otimes I_n + I_n \otimes B) = \exp(A)\otimes\exp(B),$$
valid for any $A$ and $B$. The Lie algebra (\ref{Kron_split}) naturally splits as the orthogonal  direct sum of two subspaces:
\begin{equation} \label{Kron_al1}
 \GG =  \GG_1 \oplus  \GG_2,
\end{equation}
where 
$$\GG_1 = \{A\otimes I_n\, |\, A\in \GL_n\}, \quad  \GG_2 = \{I_n\otimes B\, |\, B \in \SL_n\},$$
the first of dimension $n^2$, the second of dimension $n^2 - 1$.  The fact that $\GG_1 \perp \GG_2$
follows from the well-known property
$$ \Tr (A\otimes B) = \Tr(A) \Tr(B),$$
which in turn implies that the Frobenius inner product of a $A\otimes I_n$ and $I_n\otimes B$ is zero
if one of $A$ or $B$ has zero trace. \\

With a slight abuse of notation we can also write the splitting (\ref{Kron_al1}) as
\begin{equation}\label{Kron_al}
\GG = (\GL_n \otimes I_n) \oplus  (I_n \otimes \SL_n),
\end{equation}
stating that $\GG$ is the {\em Kronecker sum} of the Lie algebras $\GL_n$ and $\SL_n$.
This Lie algebra splitting is the infinitesimal counterpart of the Lie group decomposition (\ref{Kron_gr}). \\

Using the representation (\ref{Kron_prod}) of $G$ leads to a different splitting of the Lie algebra $\GG$, namely,
\begin{equation}\label{Kron_al}
\GG =  \GG_1 + \GG_3 = (\GL_n \otimes I_n) +  (I_n \otimes \GL_n),
\end{equation}
which is no longer a direct sum decomposition since $\GG_1 \cap \GG_3$ consists of all matrices of the form $k I_{n^2}$
with $k \in \Rr$. 
Nevertheless, this somewhat redundant representation of the tangent space of $G$ is the one which more closely corresponds
to the splitting (\ref{ADI_split}) underlying the ADI-like methods.

 	 \title{Nomenclature of decompositions, factorizations and splittings, with the associated linear systems algorithm}
\begin{table}
  \caption{Nomenclature of main decompositions, factorizations and corresponding splittings, 
  with the associated iteration.}
  \centering 
  \begin{threeparttable}
    \begin{tabular}{lccc}
	   Lie group decomp.  & Matrix factorization & Lie algebra splitting & lteration \\
     \midrule\midrule

$\begin{aligned}
& \mbox{Global Cartan-like}\\
 \end{aligned}  
  $ &   
$\begin{aligned}
&
&\mbox{Polar}\\
 \end{aligned} $   &   $\begin{aligned}
&{}\\
&\mbox{$J$-symm.~and $J$-skew symm.}\\
&\mbox{}
 \end{aligned} $ &   $\begin{aligned}
&{}\\
&\mbox{$J$-HSS}\\
&\mbox{}
 \end{aligned} $  \\
 \cmidrule(l  r ){1-4}
   $\begin{aligned}

& \mbox{Zappa-Sz\'ep}\\
& \mbox{Iwasawa (KAN)}
 \end{aligned}
$   &   
$\begin{aligned}
&\mbox{QR}\\\
&\mbox{QDR}
 \end{aligned} $   &   $\begin{aligned}
&\mbox{skew-symm.~and upper triangular}\\
&\mbox{skew-symm., diag.~and str.~upper tr.}\\  
 \end{aligned} $ &   $\begin{aligned}
 &{}\\
 &\mbox{STS}\\
&\mbox{MSTS}\\
&{}\\
 \end{aligned} $ 
 \\     
    \cmidrule(l  r ){1-4}      
$\begin{aligned}
&\mbox{(Local) Doolittle}\\
&\mbox{(Local) Crout} \\
&\mbox{(Local) LDU}
 \end{aligned}
$   &   
$\begin{aligned}
&\mbox{(LD)U}\\
&\mbox{L(DU)}\\
&\mbox{LDU} \\
\end{aligned} $   &   $\begin{aligned}
&\mbox{lower tr.~and str.~upper tr.}\\
&\mbox{upper tr.~and str.~lower tr.}\\
&\mbox{diag.~and str.~lower-upper tr.}\\
 \end{aligned} $ &   $\begin{aligned}
&{}\\
&\mbox{Gauss-Seidel}\\
&\mbox{rev.~Gauss-Seidel}\\
&\mbox{Jacobi} \\
&{}\\
 \end{aligned} $ 
 \\
\cmidrule(l  r ){1-4} 
$\begin{aligned}
&{}\\
& \mbox{Kronecker product}
 \end{aligned} $ & $\begin{aligned}
&{}\\
& \mbox{$A\otimes B$}\\
 \end{aligned} $   &   $\begin{aligned}
&{}\\
& \mbox{Kronecker sum}\\
 \end{aligned} $ &   $\begin{aligned}
&\mbox{}\\
& \mbox{ADI-like}
 \end{aligned} $  
             \\
             \\
    \midrule\midrule
    \end{tabular}
\end{threeparttable}
  \end{table}

\section{Conclusions and future work}\label{sec:7}

In this paper we have shown that many (well-known and less well-known) matrix splittings can be interpreted
as linearizations of matrix factorizations using basic notions from the theory of Lie groups and their Lie and Jordan algebras. 
We have also introduced the notion of Lie-Jordan splitting and used it to investigate a generalization of the HSS iteration,
which we call the $J$-HSS iteration, for solving  block-structured linear systems.  Possible applications of this technique include
the development of solvers and preconditioners for saddle-point problems and for complex linear systems in real equivalent form. 
Further exploration of these techniques could be the subject of future work.\\

In Table 1 we summarize the main group decompositions and corresponding Lie algebra splittings considered
in this paper.
Our list of splittings (and the corresponding factorizations) is by no means exhaustive, and many more such relationships
remain to be explored; for instance, the recent paper \cite{Edelman2} describes fifty-three different matrix factorizations,
some of which may correspond to interesting matrix splittings not previously considered. Also, the investigation
of block factorizations and the associated block splittings  is a possible avenue for future work.

\section*{Acknowledgements}  This paper is dedicated to Daniel Szyld, who has made major contributions to the theory and application of 
matrix splittings, Krylov subspace methods, preconditioning, and many other topics. The first author would like to thank Daniel for his warm friendship
and for all his advice and support over the last 30 years.  We re also grateful to two anonymous reviewers and to Angelo Vistoli for helpful
suggestions.

\section*{Funding}
The authors acknowledge funding from INdAM-GNCS (Project ``Metodi basati su matrici e tensori strutturati per problemi di algebra lineare di grandi
dimensioni"). The first author also acknowledges support from MUR (Ministero dell'Universit\`a e della Ricerca)
through the PRIN Project 20227PCCKZ
(``Low Rank Structures and Numerical Methods in Matrix and Tensor Computations and their
Applications").

\end{document}